\documentclass[12pt,a4paper]{amsart}
\usepackage[utf8]{inputenc}
\usepackage[english]{babel}
\usepackage{graphicx}
\usepackage{amsmath}
\usepackage{amssymb}
\usepackage{amscd}
\usepackage{amsthm}
\usepackage{amscd}

\usepackage[matrix,arrow,curve]{xy}
\usepackage{stmaryrd}
\usepackage{hyperref}

\usepackage{tikz}
\usetikzlibrary{arrows.meta}
\usetikzlibrary{bending}

\usepackage{graphicx}
\usepackage{xspace}
\usepackage{enumerate}
\usepackage{verbatim}
\usepackage{geometry}
\usepackage{tikz-cd}

\DeclareMathOperator{\F}{\mathbb{F}}

\address{Nikita Geldhauser, Ludwig-Maximilians-Universit\"at M\"unchen, Theresienstr. 39, 80333 M\"unchen, Germany}
\email{geldhauser@math.lmu.de}

\address{Andrei Lavrenov, Ludwig-Maximilians-Universit\"at M\"unchen, Theresienstr. 39, 80333 M\"unchen, Germany}
\email{avlavrenov@gmail.com}

\address{Victor Petrov,
Laboratory of Modern Algebra and Applications, St. Petersburg State University, 14th Line, 29b, Saint Petersburg, 199178 Russia\\
and
PDMI RAS, Fontanka 27, Saint Petersburg, 191023 Russia}
\email{victorapetrov@googlemail.com}

\address{Pavel Sechin, Universität Regensburg, Universitätsstr. 31, 93053 Regensburg, Germany}
\email{pavel.sechin@ur.de}

\begin{document}

\newcommand{\Sm}{\mathcal{S}^{\!}\mathsf{m}_k}
\newcommand{\Rings}{\mathcal{R}^{\!}\mathsf{ings}^*}
\newcommand{\rings}{\mathcal{R}^{\!}\mathsf{ings}}
\newcommand{\SmE}[1]{\mathcal{S}^{\!}\mathsf{m}_{#1}}
\newcommand{\Mot}[1]{\mathcal M^{\!}\mathsf{ot}_{#1}}
\newcommand{\Corr}[1]{\mathcal C^{\!}\mathsf{orr}_{#1}}
\newcommand{\MotF}[1]{\mathcal M_{#1}}
\newcommand{\QG}{\mathbb Q\Gamma}
\newcommand{\Inv}{\mathrm{Inv}(\QG)}
\newcommand{\Gal}{\mathrm{Gal}(k^{\mathrm{sep}}/k)}
\newcommand{\End}{\mathrm{End}}
\newcommand{\M}[1]{\mathcal{M}_{#1}}
\newcommand{\EG}{\!\,_EG}
\newcommand{\EGP}{\!\,_E(G/P)}
\newcommand{\EX}{\!\,_EX}
\newcommand{\XG}{\!\,_{\xi}G}
\newcommand{\XGP}{\!\,_{\xi}(G/P)}
\newcommand{\KQ}[1]{\mathrm K(n)^*\big(#1;\,\mathbb Q[v_n^{\pm1}]\big)}
\newcommand{\KZ}[1]{\mathrm K(n)^*\big(#1;\,\mathbb Z_{(p)}[v_n^{\pm1}]\big)}
\newcommand{\KZp}[1]{\mathrm K(n)^*\big(#1;\,\mathbb Z_p[v_n^{\pm1}]\big)}
\newcommand{\KF}[1]{\mathrm K(n)^*\big(#1;\,\mathbb F_p[v_n^{\pm1}]\big)}
\newcommand{\CHQ}[1]{\mathrm{CH}^*\big(#1;\,\mathbb Q[v_n^{\pm1}]\big)}
\newcommand{\KXZ}[1]{\!\,^{\mathrm K(n)\!}#1_{\,\mathbb Z_{(p)}[v_n^{\pm1}]}}
\newcommand{\KMotQ}{\Mot{\,\mathrm K(n)}}
\newcommand{\CHMotQv}{\Mot{\,\mathrm{CH}}}
\newcommand{\KXQ}[1]{\mathcal M_{\mathrm K(n)}(#1)}
\newcommand{\KMQ}{\mathcal M_{\,\mathrm K(n)}}
\newcommand{\CHMQv}{\MotF{\,\mathrm{CH}}}
\newcommand{\KCorrQ}{\Corr{\,\mathrm K(n)}}
\newcommand{\CHCorrQv}{\Corr{\,\mathrm{CH}}}
\newcommand{\CHCorrQ}{\Corr{\,\mathrm{CH}}}
\newcommand{\AMot}{\Mot A}

\newcommand{\e}{\varepsilon}
\newcommand{\con}{\ensuremath{\triangledown}}
\newcommand{\ra}{\ensuremath{\rightarrow}}
\newcommand{\tp}{\ensuremath{\otimes}}
\newcommand{\pr}{\ensuremath{\partial}}
\newcommand{\trigd}{\ensuremath{\triangledown}}
\newcommand{\dAB}{\ensuremath{\Omega_{A/B}}}
\newcommand{\QQ}{\ensuremath{\mathbb{Q}}\xspace}
\newcommand{\CC}{\ensuremath{\mathbb{C}}\xspace}
\newcommand{\RR}{\ensuremath{\mathbb{R}}\xspace}
\newcommand{\ZZ}{\ensuremath{\mathbb{Z}}\xspace}
\newcommand{\Zp}{\ensuremath{\mathbb{Z}_{(p)}}\xspace}
\newcommand{\Z}[1]{\ensuremath{\mathbb{Z}_{(#1)}}\xspace}
\newcommand{\NN}{\ensuremath{\mathbb{N}}\xspace}
\newcommand{\LL}{\ensuremath{\mathbb{L}}\xspace}
\newcommand{\inN}{\ensuremath{\in\mathbb{N}}\xspace}
\newcommand{\inQ}{\ensuremath{\in\mathbb{Q}}\xspace}
\newcommand{\inR}{\ensuremath{\in\mathbb{R}}\xspace}
\newcommand{\inC}{\ensuremath{\in\mathbb{C}}\xspace}
\newcommand{\OO}{\ensuremath{\mathcal{O}}\xspace}
\newcommand{\rarr}{\rightarrow}
\newcommand{\Rarr}{\Rightarrow}
\newcommand{\xrarr}[1]{\xrightarrow{#1}}
\newcommand{\larr}{\leftarrow}
\newcommand{\lrarr}{\leftrightarrows}
\newcommand{\rlarr}{\rightleftarrows}
\newcommand{\rrarr}{\rightrightarrows}
\newcommand{\al}{\alpha}
\newcommand{\bt}{\beta}
\newcommand{\ld}{\lambda}
\newcommand{\om}{\omega}
\newcommand{\Kd}[1]{\ensuremath{\Omega^{#1}}}
\newcommand{\KKd}{\ensuremath{\Omega^2}}
\newcommand{\vd}{\partial}
\newcommand{\PC}{\ensuremath{\mathbb{P}_1(\mathbb{C})}}
\newcommand{\PPC}{\ensuremath{\mathbb{P}_2(\mathbb{C})}}
\newcommand{\derz}{\ensuremath{\frac{\partial}{\partial z}}}
\newcommand{\derw}{\ensuremath{\frac{\partial}{\partial w}}}
\newcommand{\mb}[1]{\ensuremath{\mathbb{#1}}}
\newcommand{\mf}[1]{\ensuremath{\mathfrak{#1}}}
\newcommand{\mc}[1]{\ensuremath{\mathcal{#1}}}
\newcommand{\id}{\ensuremath{\mbox{id}}}
\newcommand{\dd}{\ensuremath{\delta}}
\newcommand{\bu}{\bullet}
\newcommand{\ot}{\otimes}
\newcommand{\boxt}{\boxtimes}
\newcommand{\op}{\oplus}
\newcommand{\mt}{\times}
\newcommand{\Gm}{\mathbb{G}_m}
\newcommand{\Ext}{\ensuremath{\mathrm{Ext}}}
\newcommand{\Tor}{\ensuremath{\mathrm{Tor}}}

\newcommand{\kn}[1]{\mathrm K(n)^*(#1)}
\newcommand{\ckn}[1]{\mathrm{CK}(n)^*(#1)}
\newcommand{\grckn}[1]{\mathrm{gr}_\tau^{*}\,\mathrm{CK}(n)^{*}(#1)}
\newcommand{\so}{\mathrm{SO}_m}
\newcommand{\pt}{\mathrm{pt}}

\makeatletter
\newcommand{\colim@}[2]{%
  \vtop{\m@th\ialign{##\cr
    \hfil$#1\operator@font colim$\hfil\cr
    \noalign{\nointerlineskip\kern1.5\ex@}#2\cr
    \noalign{\nointerlineskip\kern-\ex@}\cr}}%
}
\newcommand{\colim}{%
  \mathop{\mathpalette\colim@{\rightarrowfill@\textstyle}}\nmlimits@
}
\makeatother

\newtheorem{lm}{Lemma}[section]
\newtheorem{lm*}{Lemma}
\newtheorem*{tm*}{Theorem}
\newtheorem*{tms*}{Satz}
\newtheorem{tm}[lm]{Theorem}
\newtheorem{prop}[lm]{Proposition}
\newtheorem*{prop*}{Proposition}
\newtheorem{cl}[lm]{Corollary}
\newtheorem*{cor*}{Corollary}
\theoremstyle{remark}
\newtheorem*{rk*}{Remark}
\newtheorem*{rm*}{Remark}
\newtheorem{rk}[lm]{Remark}
\newtheorem*{xm}{Example}
\theoremstyle{definition}
\newtheorem{df}{Definition}
\newtheorem*{nt}{Notation}
\newtheorem{Def}[lm]{Definition}%[section]
\newtheorem*{Def-intro}{Definition}
\newtheorem{Rk}[lm]{Remark}
\newtheorem{Ex}[lm]{Example}

\theoremstyle{plain}
\newtheorem{Th}[lm]{Theorem}
\newtheorem*{Th*}{Theorem}
\newtheorem*{Th-intro}{Theorem}
\newtheorem{Prop}[lm]{Proposition}
\newtheorem*{Prop*}{Proposition}
\newtheorem{Cr}[lm]{Corollary}
\newtheorem{Lm}[lm]{Lemma}
\newtheorem{Conj}{Conjecture}
\newtheorem*{BigTh}{Classification of Operations Theorem  (COT)}
\newtheorem*{BigTh-add}{Algebraic Classification of Additive Operations Theorem  (CAOT)}

\tikzcdset{
arrow style=tikz,
diagrams={>={Straight Barb[scale=0.8]}}
}

\title{Morava $J$-invariant}
\author{Nikita Geldhauser, Andrei Lavrenov,\\ Victor Petrov, Pavel Sechin}
\maketitle

\begin{abstract}
We compute the co-multiplication of the algebraic Morava K-theory for split orthogonal groups. This allows us to compute the decomposition of the Morava motives of generic maximal orthogonal Grassmannians and to compute a Morava K-theory analogue of the $J$-invariant in terms of the ordinary (Chow) $J$-invariant.
\end{abstract}

\section*{Introduction}

Motives were introduced by Alexander Grothendieck in the 1960s and since then they became a fundamental tool for investigating the structure of algebraic varieties. In particular, the study of Chow motives of twisted flag varieties has led to the solution of several classical problems in the theory of algebraic groups. More precisely, it was the $J$-invariant that played an important role in the progress on the Kaplansky problem about possible values of the $u$-invariant of fields by Vishik~\cite{Vuinv}, in the solution of a problem of Serre about groups of type $\mathsf E_8$~\cite{S16, GS10}, and a conjecture of Rost about groups of type $\mathsf{E}_7$~\cite{GPS16}.

The $J$-invariant is a discrete invariant which describes the motivic behavior of the variety of Borel subgroups of a semisimple linear algebraic group. It was first introduced by Vishik in~\cite{V-gras} in the case of orthogonal groups, and then generalized in~\cite{PSZjinv} to other algebraic groups.

Recently, the first and the third authors generalized in~\cite{PShopf} the $J$-invariant to other oriented cohomology theories in the sense of Levine--Morel~\cite{LM}. Their approach gives a ``categorified'' version of the invariant: for an oriented cohomology theory $A^*$ and for a split semisimple algebraic group $G$ the corresponding $J$-invariant is a certain quotient bi-algebra of $A^*(G)$. Using the Milnor--Moore Theorem on the classification of Hopf algebras one can show that the new definition coincides with the old one in the case $A^*=\mathrm{CH}^*(-;\,\mathbb F_p)$.

Thus, to understand the $J$-invariant for general oriented cohomology theories one has to analyze the structure of the cohomology $A^*(G)$ of a split group. In the present article we focus on the algebraic Morava K-theory of orthogonal groups. Note that the algebraic Morava K-theory and Morava motives have a deep connection with Galois cohomology as was already recognized by Voevodsky in his program of the proof of the Bloch--Kato conjecture~\cite{Vo95} and further investigated in~\cite{SeSe}. The algebraic Morava K-theory has also applications to the structure of torsion in Chow rings of twisted flag varieties (see \cite{SeSe}). Besides, the algebraic Morava K-theory plays an important role in the study of isotropic motives~\cite{Vi22,Vi24} and in particular the description of the Balmer spectrum of the Morel-Voevodsky $\mathbb A^1$-stable homotopy category~\cite{ViDu}.

The computation of the cohomology of Lie groups was a classical problem in algebraic topology dating back to Cartan, which gave rise to the theory of Hopf algebras. A uniform description of the cohomology rings of simple compact Lie groups, and of the Chow rings of corresponding algebraic groups can be found in~\cite{Kac}. The computation of the co-multiplication was finished in~\cite{IKT}.

The topological K-theory of simple Lie groups was computed by Atiyah and Hirzebruch, Hodgkin~\cite{Hod}, Araki~\cite{Ara}, and the (higher) algebraic K-theory of algebraic groups was computed by Levine~\cite{L-kt}, Merkurjev~\cite{M-kt} and others (also for twisted forms).

In turn, the topological Morava K-theory of simple compact Lie groups is not known in many cases. Some cases were computed by Yagita~\cite{Yag1,Yag2}, Rao~\cite{R90,R97,R12}, Nishimoto~\cite{Nis}, Mimura~\cite{MN}, and many others (see, e.g.,~\cite{HMNS}). In particular, the topological Morava K-theory of orthogonal (and spinor) groups is known only additively by~\cite{R90,Nis}, but the multiplication and co-multiplication is not known (see~\cite{R97,R08,R12} for partial results). 
 
However, the algebraic Morava K-theory seems to behave much better. In our previous article~\cite{LPSS} we described $\mathrm K(n)^*(\mathrm{SO}_m)$ {\it as an algebra}. Combining our technique with computations of Rao~\cite{R97} we prove the following

\begin{tm*}[Theorem~\ref{co-mult}]
The algebra structure of $\mathrm K(n)^*(\mathrm{SO}_m)$ is given by
$$
\mathbb F_2[v_n^{\pm1}][e_1,e_2,\ldots,e_s]/(e_i^2-e_{2i})
$$
where $s=\mathrm{min}\left(\lfloor\frac{m-1}{2}\rfloor,\,2^{n}-1\right)$ and $e_{2i}$ stands for $0$ if $2i>s$, cf.~\cite[Theorem~6.13]{LPSS}. The reduced co-multiplication $\widetilde\Delta(x)=\Delta(x)-x\otimes1-1\otimes x$ is given by 
$$
\widetilde\Delta(e_{\langle2k\rangle})=\sum_{i=0}^{\nu_2(k)}v_n^{i+1}\,e_{\langle k/2^{i}\rangle}\otimes e_{\langle k/2^{i}\rangle}\,\prod_{j=0}^{i-1}\left(e_{\langle k/2^{j}\rangle}\otimes1+1\otimes e_{\langle k/2^{j}\rangle}\right)
$$
where $\langle t\rangle$ stands for $2^n-1-t$, $k>0$, and $\nu_2(k)$ is the $2$-adic valuation of $k$, i.e., $k/2^{\,\nu_2(k)}$ is an odd integer, and
$$
\widetilde\Delta(e_{2^n-1})=v_n\,e_{2^{n}-1}\otimes e_{2^{n}-1}.
$$
\end{tm*}

In other words, for a fixed $n$ the algebra and the co-algebra structures of the $n$-th Morava K-theory of orthogonal groups  $\mathrm{SO}_m$ stabilize for $m\geq 2^{n+1}-1$. In the range $m\leq 2^{n+1}$ the algebra structure of the Morava K-theory coincides with the algebra structure of the Chow theory, but the co-algebra structure of the Morava K-theory coincides with the co-algebra structure of the Chow theory only for  $m\leq 2^{n}$. In the intermediate range $2^n< m\leq 2^{n+1}$ the co-algebra structure is given by the theorem above. 

An important step in our arguments is a comparison of the algebraic and topological Morava K-theories. Recall that for a compact connected Lie group $K$, and the corresponding reductive group $G=K_{\mathbb C}$ there is a natural map
$$
\Omega^*(G)\rightarrow\mathrm{MU}^*(K)
$$
from the algebraic cobordism of Levine--Morel of $G$ to the complex cobordism of $K$. Yagita conjectured in~\cite{Yag} that this map is injective and obtained partial results in this direction, however, his conjecture remains open. On the other hand, it is natural to consider an analogue of this conjecture for the Morava K-theory. %seems to be much easier.
Using our computations of $\mathrm K(n)^*(\mathrm{SO}_m)$ we show in Section~\ref{section-yagita} that the natural map
$$
\mathrm K(n)^*(\mathrm{SO}_m)\rightarrow\mathrm K(n)^*_{\mathrm{top}}\big(\mathrm{SO}(m)\big)
$$
is indeed injective. Next, in Section~\ref{topology} we use the computations of~\cite{R97} to deduce the co-algebra structure of $\mathrm K(n)^*(\mathrm{SO}_m)$ for all $n$ and $m$. 

In the last part of the article we apply our computations of the algebraic Morava K-theory of orthogonal groups and provide a complete motivic decomposition of connected components of all generic maximal orthogonal Grassmannians: 

\begin{tm*}[Corollary~\ref{application}]
Let $n\in\mathbb N\setminus0$, and let $q$ be a quadratic form of dimension $m$ with trivial discriminant and denote by $J(q)$ its {\rm(}classical Chow{\rm)} $J$-invariant~\cite[Definition~5.11]{V-gras}. Assume that 
$$
J(q)\cap\{1,\ldots,2^n-1\}=\emptyset
$$
{\rm(}e.g., this holds for a generic quadratic form $q${\rm)}. Let $X$ denote a connected component of its maximal orthogonal Grassmannian {\rm(}see Section~\ref{section-grass}{\rm)}. Then the following holds.
\begin{enumerate}
\item
The $\mathrm K(n)$-motive $\mathcal M_{\mathrm K(n)}(X)$ of $X$ is indecomposable for $m\leq2^{n+1}-2$.
\item
The $\mathrm K(n)$-motive $\mathcal M_{\mathrm K(n)}(X)$ of $X$ has $2^{\lfloor\frac{m-1}{2}\rfloor-2^n+2}$ indecomposable summands of rank $2^{2^n-2}$ for $m\geq2^{n+1}-1$.
\end{enumerate}
\end{tm*}

Note that our previous results describing the structure of Morava motives of maximal orthogonal Grassmannians~\cite{LPSS} did not use the co-algebra structure of $\mathrm K(n)^*({\so})$ and gave only a partial decomposition of the Morava motive.

\section{Morava K-theory and related (co)homology theories}

\subsection{Algebraic cohomology theories}

We work over a field $k$ with $\mathrm{char}(k)=0$, and denote by $\mathcal S\mathsf m_k$ the category of smooth quasi-projective varieties over $k$. We often denote $\mathrm{pt}=\mathrm{Spec}(k)$. 

We consider oriented cohomology theories on $\mathcal S\mathsf m_k$ as defined by Levine--Morel~\cite[Definition~1.1.2]{LM}. In order to distinguish these from generalized cohomology theories of topological spaces (discussed below) we sometimes call them {\it algebraic} cohomology theories. Let $\Omega^*$ denote the algebraic cobordism of Levine--Morel~\cite{LM}, i.e., the universal oriented cohomology theory~\cite[Theorem~1.2.6]{LM}. 

Let $\mathbb L$ be the Lazard ring classifying formal group laws (see~\cite[Definition~A2.1.1]{Rav}). In other words, for any commutative ring $R$ the formal group laws over $R$ are in one-to-one correspondence with the ring homomorphisms $\mathbb L\rightarrow R$~\cite[Theorem~A2.1.8]{Rav}. The Lazard ring $\mathbb L$ is isomorphic to a polynomial ring $\mathbb Z[x_1,x_2,\ldots]$ of infinitely many variables~\cite[Theorem~A2.1.10]{Rav}. By~\cite[Theorem~1.2.7]{LM}, $\Omega^*(\mathrm{pt})\cong\mathbb L$.
 
 We would like to warn the reader that $\mathbb L$ is graded in~\cite{Rav} in such a way that $|x_i|=2i$, however, using the grading induced from $\Omega^*(\mathrm{pt})$ we have $|x_i|=-i$. In the present paper we stick to the latter convention. If $R$ is a commutative graded ring, then the graded homomorphisms $\mathbb L\rightarrow R$ classify the formal group laws $F$ which are homogeneous of degree $1$ as elements of $R[\![x,y]\!]$.

Given a commutative, graded $\mathbb L$-algebra $R$ it is easy to see that $A^*=\Omega^*(-)\otimes_{\mathbb L}R$ also has a natural structure of an oriented cohomology theory. If $F$ is a formal group law classified by $\mathbb L\rightarrow R$, we will say that $A^*$ is a {\it free theory} corresponding to $F$ and sometimes denote $F_A=F$.

\begin{Ex}
\label{algebraic-theories}
\ 
\begin{enumerate}
\item
For the localization $\mathbb L_{(2)}=\mathbb L\otimes_{\mathbb Z}\mathbb Z_{(2)}$ at the prime ideal $(2)=2\mathbb Z\trianglelefteq\mathbb Z$ and for the natural map $\mathbb L\rightarrow\mathbb L_{(2)}$ we denote the corresponding free theory by 
$$
\Omega_{(2)}^*=\Omega^*(-)\otimes_{\mathbb L}\mathbb L_{(2)}=\Omega^*(-)\otimes_{\mathbb Z}\mathbb Z_{(2)}.
$$
\item
\label{alg-bp}
Let $V=\mathbb Z_{(2)}[v_1,v_2,\ldots]$ with $|v_i|=1-2^i$ and let $F$ denote the universal $2$-typical formal group law over $V$ of~\cite[Theorem~A2.1.25]{Rav} (we remark again that $v_i$ are graded in a different way in~\cite{Rav}). We remark that the map classifying $F$ induces a surjective map $\mathbb L_{(2)}\twoheadrightarrow V$ by~\cite[Theorem~A2.1.25]{Rav}. The corresponding free theory
$$
\mathrm{BP}^*=\Omega^*(-)\otimes_{\mathbb L}\mathbb Z_{(2)}[v_1,v_2,\ldots]=\Omega_{(2)}^*(-)\otimes_{\mathbb L_{(2)}}V
$$
is the {\it algebraic} Brown--Peterson cohomology theory. One can similarly define Brown--Peterson theories for odd primes $p$, however in this article we will only consider the case $p=2$.
\item\label{ex113}
For a fixed natural number $n>0$ let $I(2,\,n)=(2, v_1,\ldots, v_{n-1})$ be an ideal in $V$, and consider the composite map $\mathbb L\rightarrow V\rightarrow V/I(2,\,n)=\mathbb F_2[v_n,v_{n+1},\ldots]$. Then the algebraic $\mathrm{P}(n)^*$ theory is by definition 
$$
\mathrm P(n)^*=\Omega^*(-)\otimes_{\mathbb L}\mathbb F_2[v_n,v_{n+1},\ldots]=\mathrm{BP}^*(-)\otimes_{V}V/I(2,\,n).
$$
\item
Consider the composition of $\mathbb L\rightarrow\mathbb F_2[v_n,v_{n+1},\ldots]$ with the natural quotient map $\mathbb F_2[v_n,v_{n+1},\ldots]\rightarrow\mathbb F_2[v_n]$. Then the corresponding free theory
$$
\mathrm{CK}(n)^*=\Omega^*(-)\otimes_{\mathbb L}\mathbb F_2[v_n]=\mathrm P(n)^*(-)\otimes_{V/I(2,\,n)}\mathbb F_2[v_n]
$$
is the algebraic {\it connective} Morava K-theory.
\item
Consider the composition of $\mathbb L\rightarrow\mathbb F_2[v_n]$ with the natural localization map $\mathbb F_2[v_n]\rightarrow\mathbb F_2[v_n^{\pm1}]$. Then the corresponding free theory
$$
\mathrm{K}(n)^*=\Omega^*(-)\otimes_{\mathbb L}\mathbb F_2[v_n^{\pm1}]=\mathrm{CK}(n)^*(-)\otimes_{\mathbb F_2[v_n]}\mathbb F_2[v_n^{\pm1}]
$$ 
is the algebraic {\it periodic} Morava K-theory which we will usually call simply the {\it Morava K-theory}.
\item
Consider the composition of $\mathbb L\rightarrow\mathbb F_2[v_n]$ with the natural quotient map $\mathbb F_2[v_n]\rightarrow\mathbb F_2$. Then the corresponding free theory
$$
\mathrm{Ch}^*=\mathrm{CH}^*(-;\,\mathbb F_2)=\Omega^*(-)\otimes_{\mathbb L}\mathbb F_2=\mathrm{CK}(n)^*(-)\otimes_{\mathbb F_2[v_n]}\mathbb F_2
$$
is the usual Chow theory with $\mathbb F_2$-coefficients by~\cite[Theorem~1.2.19]{LM}.
\end{enumerate}
\end{Ex}

In the next subsection we consider topological versions of these theories.

\subsection{Topological (co)homology theories}

We refer the reader to~\cite{adams} for a general introduction to generalized (co)homology theories, and to~\cite[Chapter~4]{Rav} for a survey of different theories related to the Morava K-theory. A more detailed exposition of these topics can be found in~\cite{Rud}.

For a spectrum $E$ (CW-spectrum in terminology of~\cite[Part~III, Chapter~2]{adams}) and a CW-complex $X$ we consider (unreduced) generalized homology and cohomology theories
$$
E_*(X)=[\mathbb S,\,E\wedge\Sigma^\infty X_+]_*\quad\text{and}\quad E^*(X)=[\Sigma^\infty X_+,\,E]_{-*}
$$
represented by $E$ (see~\cite[Part~III, Chapter~6]{adams}), where $X_+=X\sqcup*$, $\Sigma^\infty$ denotes the suspension spectrum functor and $\mathbb S$ is the sphere spectrum.

\begin{Ex}\ 
\begin{enumerate}
\item
Let $\mathrm{MU}$ denote the complex cobordism spectrum (see, e.g.,~\cite[Chapter~IV, Section~1]{Rav} or~\cite[Chapter~VII]{Rud}). By Quillen's Theorem~\cite[Theorem~4.1.6]{Rav}, the complex cobordism of a point $*$ is canonically isomorphic to the Lazard ring $\mathbb L$, and the isomorphism doubles the grading, i.e., $\mathrm{MU}^{2k+1}(*)=0$ and $\mathrm{MU}^{2k}(*)=\mathbb L^k$. 
\item
We can consider the {\it $2$-local complex cobordism}, i.e., the spectrum 
$$
\mathrm{MU}_{(2)}=\mathrm{MU}\wedge\mathbb S_{(2)}
$$ 
where $\mathbb S_{(2)}$ is the Moore spectrum of type $\mathbb Z_{(2)}$ (see~\cite[Part~III, Example~6.6]{adams}). The spectrum $\mathrm{MU}_{(2)}$ inherits the canonical ring structure from $\mathrm{MU}$ (see, e.g.,~\cite[Chapter~II, Theorem~5.15]{Rud}). We also remark that for a finite CW-complex $X$ we have
$$
\mathrm{MU}_{(2)}^*(X)=\mathrm{MU}^*(X)\otimes_{\mathbb Z}\mathbb Z_{(2)}
$$
(see~\cite[Part~III, Proposition~6.9]{adams}), in particular, $\mathrm{MU}_{(2)}^*(*)\cong\mathbb L_{(2)}$.
\item
We will denote $\mathrm{BP}_{\mathrm{top}}$ (or also $\mathrm{BP}^{\mathrm{top}}$) the Brown--Peterson spectrum, see~\cite[Theorem~4.1.12]{Rav} or~\cite[Chapter~VII, Definition~3.20]{Rud}. It is an associative commutative ring spectrum and a natural map
\begin{equation}
\label{mubp}
\mathrm{MU}_{(2)}\rightarrow\mathrm{BP}_{\mathrm{top}}
\end{equation}
of ring spectra is a retraction. There is a natural isomorphism $V\cong\mathrm{BP}_{\mathrm{top}}^*(*)$ (which doubles the grading), and the map~(\ref{mubp}) evaluated on $*$ induces a natural map $\mathbb L_{(2)}\twoheadrightarrow V$, see Example~\ref{algebraic-theories}\,(\ref{alg-bp}). We will refer to $\mathrm{BP}^{\mathrm{top}}_*$ and $\mathrm{BP}_{\mathrm{top}}^*$ as the {\it topological} Brown--Peterson (co)homology theories.
\item
Let us also denote by $\mathrm P(n)_{\mathrm{top}}$ (or $\mathrm P(n)^{\mathrm{top}}$) the spectrum defined by killing $I(2,n)$ in the spectrum $\mathrm{BP}_{\mathrm{top}}$ (see Example~\ref{algebraic-theories}~(\ref{ex113})). In other words, we put $\mathrm P(0)_{\mathrm{top}}=\mathrm{BP}_{\mathrm{top}}$ and $\mathrm P(n+1)_{\mathrm{top}}$ is the cofiber of the map 
$$
\Sigma^{2(2^n-1)}\mathrm P(n)_{\mathrm{top}}\stackrel{v_n}{\longrightarrow}\mathrm P(n)_{\mathrm{top}}
$$ 
(see~\cite[Chapter~4, Section~2]{Rav} and also~\cite{nassau-diplom}). The ring $\mathrm P(n)_{\mathrm{top}}^*(*)$ is naturally isomorphic to $\mathbb F_2[v_n,v_{n+1},\ldots]$.
\item
Killing the ideal $(v_{n+1},v_{n+2},\ldots)$ in the spectrum $\mathrm P(n)_{\mathrm{top}}$ we obtain the connected Morava K-theory spectrum $\mathrm{k}(n)$, which we will prefer to denote $\mathrm{CK}(n)_{\mathrm{top}}$ or $\mathrm{CK}(n)^{\mathrm{top}}$ in this paper, see~\cite[Chapter~4, Section~2]{Rav}. The ring $\mathrm{CK}(n)_{\mathrm{top}}^*(*)$ is isomorphic to $\mathbb F_2[v_n]$. We refer to $\mathrm{CK}(n)^{\mathrm{top}}_*$ and $\mathrm{CK}(n)_{\mathrm{top}}^*$ as the {\it topological} connective Morava K-theory (co)homology.
\item
One defines the Morava K-theory spectrum $\mathrm K(n)_{\mathrm{top}}$ or $\mathrm K(n)^{\mathrm{top}}$ as the directed colimit of the sequence
$$
\mathrm{k}(n)\stackrel{v_n}{\longrightarrow}\Sigma^{2(1-2^n)}\mathrm{k}(n)\stackrel{v_n}{\longrightarrow}\Sigma^{4(1-2^n)}\mathrm{k}(n)\stackrel{v_n}{\longrightarrow}\ldots
$$
see~\cite[Chapter~4, Section~2]{Rav}. We refer to $\mathrm{K}(n)^{\mathrm{top}}_*$ and $\mathrm{K}(n)_{\mathrm{top}}^*$ as the {\it topological} (periodic) Morava K-theory (co)homology. We remark that $\mathrm K(n)^*_{\mathrm{top}}(X)$ coincides with the ring theoretic localization of $\mathrm{CK}(n)^*_{\mathrm{top}}$ by the powers of $v_n$ (cf.~\cite[Section~2]{JW}) for a finite CW-complex $X$, in particular $\mathrm K(n)^*_{\mathrm{top}}(*)\cong\mathbb F_2[v_n^{\pm1}]$.
\item
The Eilenberg--MacLane spectrum $\mathrm H\mathbb Z/2$ of type $\mathbb Z/2$ can in fact be obtained from $\mathrm{BP}_{\mathrm{top}}$ by killing $(2,v_1,v_2,\ldots)$ (see~\cite[Chapter~4]{Rav}). The corresponding (co)homology theories are the usual (co)homology $\mathrm H_*(-;\,\mathbb F_2)$ and $\mathrm H^*(-;\,\mathbb F_2)$ with $\mathrm{mod}\ 2$ coefficients.
\end{enumerate}
\end{Ex}

Observe that $\mathrm P(n)_{\mathrm{top}}$ does not admit a commutative multiplicative structure (for $p=2$), however it has exactly two multiplicative structures
\begin{equation}
\label{mu12}
\mu_1,\,\mu_2\colon\mathrm P(n)_{\mathrm{top}}\wedge\mathrm P(n)_{\mathrm{top}}\rightarrow\mathrm P(n)_{\mathrm{top}}
\end{equation}
for which $\mathrm{P}(n)_{\mathrm{top}}$ is a $\mathrm{BP}_{\mathrm{top}}$-algebra spectrum, and, moreover, $\mu_1=\mu_2\circ\tau$ where $\tau$ denotes the switch $\tau\colon X\wedge Y\rightarrow Y\wedge X$ (see~\cite[Proposition~2.4]{Wue} and~\cite{nassau-diplom,nassau-paper}). Therefore the algebra structure of $\mathrm P(n)^*_{\mathrm{top}}(X)$ with respect to any of $\mu_i$ determines the other one. The same is true for the multiplicative structures on $\mathrm{CK}(n)_{\mathrm{top}}$ and $\mathrm{K}(n)_{\mathrm{top}}$, cf.~\cite[Remark~2.6]{Wue}, see also~\cite{strickland}.

\subsection{Comparison of algebraic and topological cohomology theories}

Consider a field extension $k\subseteq\mathbb C$. Then there is a natural map from the discussed (algebraic) oriented cohomology theories to their topological versions. 

Observe that for $X\in\mathcal S\mathsf m_k$ we obtain an oriented cohomology theory 
$$
X\mapsto\mathrm{MU}^{2*}(X(\mathbb C)),
$$
see~\cite[Example~1.2.10]{LM}. Then by the universality of $\Omega^*$ we obtain a canonical map
$$
\Omega^*(X)\rightarrow\mathrm{MU}^{2*}(X(\mathbb C)).
$$
Since $\mathrm{MU}^{*}_{(2)}(X(\mathbb C))$ has the structure of an $\mathbb L_{(2)}$-module, localizing at $(2)$ we also obtain a morphism of oriented cohomology theories $\Omega^*_{(2)}\rightarrow\mathrm{MU}^*_{(2)}$.

For $A=\mathrm{BP},\,\mathrm{P}(n),\,\mathrm{CK}(n)$ consider the diagram
$$\xymatrix{
\Omega_{(2)}^*(X)\ar[r]^{}\ar[d]&\mathrm{MU}_{(2)}^*(X(\mathbb C))\ar[d]\\
A^*(X)&A^*_{\mathrm{top}}(X(\mathbb C)).
}$$
Since the generators of $\mathrm{Ker}(\mathbb L_{(2)}\twoheadrightarrow A^*(\mathrm{pt}))$ act as zero on $A^*_{\mathrm{top}}(X(\mathbb C))$, we conclude that the top horizontal map induces a natural map from $A^*(X)$ to $A^*_{\mathrm{top}}(X(\mathbb C))$. Passing to the localization, we also obtain a natural map from $\mathrm{K}(n)^*(X)$ to $\mathrm{K}(n)^*_{\mathrm{top}}(X(\mathbb C))$.

Moreover, since the natural map $\mathrm{BP}^*(X)\rightarrow\mathrm{BP}_{\mathrm{top}}^*(X(\mathbb C))$ respects the multiplication, and $\mu_i$ from~(\ref{mu12}) are compatible with the multiplicative structure on $\mathrm{BP}_{\mathrm{top}}$, we conclude that the natural map $\mathrm P(n)^*(X)\rightarrow\mathrm P(n)^*_{\mathrm{top}}(X(\mathbb C))$ is a ring homomorphism for any choice of $\mu_i$ (cf.~\cite[(2.3)\.(i)]{Wue} and~\cite{nassau-diplom,nassau-paper}). The same statement remains true for the multiplicative structures on $\mathrm{CK}(n)_{\mathrm{top}}$ and $\mathrm{K}(n)_{\mathrm{top}}$, cf.~\cite[Remark~2.6]{Wue}.

We say that $X\in\mathcal S\mathsf m_k$ is {\it cellular} if $X$ has a finite filtration 
$$
\emptyset = X_{-1}\subseteq X_0\subseteq X_1\subseteq\ldots\subseteq X_n = X
$$
by closed subvarieties such that the $X_i\setminus X_{i-1}$ are isomorphic to a disjoint union of affine spaces $\mathbb A^{d_i}$ for all $i = 0,\,\ldots, n$. We will need the following proposition.

\begin{prop}
\label{HK}
Let $X\in\mathcal S\mathsf m_k$ be a cellular variety, and let $A$ be one of the $\mathrm{BP}$, $\mathrm P(n)$, $\mathrm{CK}(n)$, $\mathrm{K}(n)$. Then the natural map
$$
A^*(X)\rightarrow A^*_{\mathrm{top}}(X(\mathbb C))
$$
is an isomorphism. 
\end{prop}
\begin{proof}
The proof of~\cite[Theorem~6.1]{HK} can be repeated to obtain the claim for $A=\mathrm{BP},\,\mathrm P(n),\,\mathrm{CK}(n)$. Taking the localization by the powers of $v_n$ we conclude that
$$
\mathrm{K}(n)^*(X)\rightarrow\mathrm{K}(n)^*_{\mathrm{top}}(X(\mathbb C))
$$
is also an isomorphism.
\end{proof}

\subsection{Atiyah--Hirzebruch spectral sequences}

For a finite CW-complex $X$ its filtration by $r$-skeletons $X^r$
$$
\emptyset = X^{-1}\subseteq X^0\subseteq X^1\subseteq\ldots\subseteq X^n = X
$$
induces the Atiyah--Hirzebruch spectral sequences
$$
\mathrm H_p(X;\,E_q(*))\Longrightarrow E_{p+q}(X)\ \text{ and }\ \mathrm H^p(X;\,E_q(*))\Longrightarrow E^{p+q}(X)
$$
for any spectrum $E$, see~\cite[Part~III, Chapter~7]{adams}. 

We remark that the morphism of spectra 
$
\mathrm{CK}(n)_{\mathrm{top}}\rightarrow\mathrm K(n)_{\mathrm{top}}
$  
induces a morphism of the Atiyah--Hirzebruch spectral sequences for $\mathrm{CK}(n)_{\mathrm{top}}^*(X)$ and $\mathrm K(n)_{\mathrm{top}}^*(X)$. On the second page this map coincides with the localization, in particular, it is injective. As a result, if the Atiyah--Hirzebruch spectral sequence for $\mathrm{K}(n)_{\mathrm{top}}^*(X)$ collapses at the second page, we see that the Atiyah--Hirzebruch spectral sequence for $\mathrm{CK}(n)_{\mathrm{top}}^*(X)$ also collapses. In particular, the following holds.

\begin{prop}
\label{ahss-cohomological}
Let $X$ be a finite CW-complex, $n\in\mathbb N$, and assume that the Atiyah--Hirzebruch spectral sequence $\mathrm H^p\big(X;\,\mathrm K(n)^{\mathrm{top}}_q(*)\big)\Rightarrow \mathrm K(n)_{\mathrm{top}}^{p+q}(X)$ collapses at the second page. 

Then $\mathrm{CK}(n)_{\mathrm{top}}^*(X)$ is a free graded $\mathbb F_2[v_n]$-module of rank $\mathrm{dim}_{\mathbb F_2\,}\mathrm{H}^*(X;\,\mathbb F_2)$ and the natural map $\mathrm{CK}(n)_{\mathrm{top}}\rightarrow\mathrm H\mathbb Z/2$ induces an isomorphism
$$
\mathrm{CK}(n)_{\mathrm{top}}^*(X)\otimes_{\mathbb F_2[v_n]}\mathbb F_2=\mathrm{H}^*(X;\,\mathbb F_2).
$$
\end{prop}

We  will also need the following result proven in~\cite[Section~4]{JW}.

\begin{prop}
\label{ahss-homological}
Let $X$ be a finite CW-complex, $n\in\mathbb N\setminus0$. Then the following are equivalent.
\begin{enumerate}[{\rm (1)}]
\item
The Atiyah--Hirzebruch spectral sequence 
$$
\mathrm H_p(X;\,\mathrm P(n)^{\mathrm{top}}_q(*))\Longrightarrow \mathrm P(n)^{\mathrm{top}}_{p+q}(X)
$$ 
collapses {\rm(}at the second page{\rm)}.
\item
$\mathrm{P}(n)^{\mathrm{top}}_{*}(X)$ is a free $\mathbb F_2[v_n,v_{n+1},\ldots]$-module of rank $\mathrm{dim}_{\mathbb F_2\,}\mathrm H_*(X;\,\mathbb F_2)$ and the canonical map $\mathrm P(n)^{\mathrm{top}}\rightarrow\mathrm H\mathbb Z/2$ induces an isomorphism
$$
\mathrm P(n)^{\mathrm{top}}_{*}(X)\otimes_{\mathbb F_2[v_n,v_{n+1},\ldots]}\mathbb F_2\cong\mathrm H_*(X;\mathbb F_2).
$$
\item
The canonical map $\mathrm P(n)^{\mathrm{top}}_{*}(X)\rightarrow\mathrm H_*(X;\mathbb F_2)$ is surjective.
\item
The Atiyah--Hirzebruch spectral sequence 
$$
\mathrm H_p(X;\,\mathrm{CK}(n)^{\mathrm{top}}_q(*))\Longrightarrow \mathrm{CK}(n)^{\mathrm{top}}_{p+q}(X)
$$ collapses.
\item
$\mathrm{CK}(n)^{\mathrm{top}}_{*}(X)$ is a free $\mathbb F_2[v_n]$-module of rank $\mathrm{dim}_{\mathbb F_2\,}\mathrm H_*(X;\,\mathbb F_2)$ and the natural map $\mathrm{CK}(n)^{\mathrm{top}}\rightarrow\mathrm H\mathbb Z/2$ induces an isomorphism
$$
\mathrm{CK}(n)^{\mathrm{top}}_{*}(X)\otimes_{\mathbb F_2[v_n]}\mathbb F_2\cong\mathrm H_*(X;\mathbb F_2).
$$
\item
The canonical map $\mathrm{CK}(n)^{\mathrm{top}}_{*}(X)\rightarrow\mathrm H_*(X;\mathbb F_2)$ is surjective.
\end{enumerate}
If the equivalent conditions {\rm (1)--(6)} hold then one also has the following.
\begin{enumerate}[{\rm (1)}]
\setcounter{enumi}{6}
\item
The canonical map $\mathrm P(n)^{\mathrm{top}}\rightarrow\mathrm{CK}(n)^{\mathrm{top}}$ induces an isomorphism
$$
\mathrm P(n)^{\mathrm{top}}_{*}(X)\otimes_{\mathbb F_2[v_n,v_{n+1},\ldots]}\mathbb F_2[v_n]\cong\mathrm{CK}(n)^{\mathrm{top}}_*(X).
$$
\item
There is a canonical isomorphism
$$
\mathrm{CK}(n)^*_{\mathrm{top}}(X)\cong\mathrm{Hom}_{\mathbb F_2[v_n]}(\mathrm{CK}(n)^{\mathrm{top}}_*(X),\,\mathbb F_2[v_n]).
$$
\end{enumerate}

\end{prop}
\begin{proof}
The implications 
$$
(1)\Rightarrow(2)\Rightarrow(3),\quad(1)\Rightarrow(4)\Rightarrow(5)\Rightarrow(6)
$$
are straightforward. For $(3)\Rightarrow(1)$ see~\cite[Proposition~1.1]{R93} (cf.~\cite[Proof of Theorem~4.16]{JW}). 

Assume that $(6)$ holds. Then by~\cite[Theorem~4.8 and Remark~4.9]{JW} one has that the natural map $\mathrm P(n)^{\mathrm{top}}_*(X)\rightarrow\mathrm{CK}(n)^{\mathrm{top}}_*(X)$ is surjective. In particular, $(6)\Rightarrow(3)$. This finishes the proof of equivalence (1)--(6). Moreover, we see that the induced surjection 
$$
\mathrm P(n)^{\mathrm{top}}_{*}(X)\otimes_{\mathbb F_2[v_n,v_{n+1},\ldots]}\mathbb F_2[v_n]\twoheadrightarrow\mathrm{CK}(n)^{\mathrm{top}}_*(X)
$$
is in fact an isomorphism by~(2) and~(5) comparing the dimensions. Therefore $(6)\Rightarrow(7)$. Finally, for $(4)\Rightarrow(8)$ see~\cite[Lemma~13.9]{adams}.
\end{proof}

We remark that for the {\it periodic} topological Morava K-theory there is a canonical isomorphism
$$
\mathrm K(n)^*_{\mathrm{top}}(X)\cong\mathrm{Hom}_{\mathbb F_2[v_n^{\pm1}]}\big(\mathrm K(n)^{\mathrm{top}}_*(X),\,\mathbb F_2[v_n^{\pm1}]\big)
$$
for any finite CW-complex $X$, see, e.g.,~\cite[Section~8]{Wue-odd}.

\section{An analogue of Yagita's conjecture for Morava K-theory}
\label{section-yagita}

\subsection{Yagita's conjecture}
\label{yagita-conjecture}

Let $K$ be a compact connected Lie group, $T$ its maximal torus, and $\pi\colon K\rightarrow K/T$ the natural projection. Denote by $G=K_{\mathbb C}$ the corresponding (split) reductive group over $\mathbb C$. Kac observes in~\cite{Kac} that the results of Grothendieck~\cite{Gro} imply that
$$
\mathrm{CH}^*(G;\,\mathbb F_p)\cong\pi^*\big(\mathrm H^*(K/T;\,\mathbb F_p)\big)
$$
where the latter is a subring in $\mathrm H^*(K;\,\mathbb F_p)$ (and the isomorphism doubles the grading).

In the article~\cite{Yag} Yagita conjectured that the same is true for the cobordism theories. Recall that $\mathrm{MU}^*$ denotes the complex cobordism theory, and $\Omega^*$ denotes the algebraic cobordism theory of Levine--Morel. Yagita shows in~\cite[Theorem~1.1]{Yag} that for the ideal $I=(p,x_{p-1},x_{p^2-1},\ldots)$ in $\mathbb L\cong\mathbb Z[x_1,x_2,\ldots]$ and for a simply-connected compact Lie group $K$ one has an isomorphism
$$
\Omega^*(G)/I^2\cong\pi^*\big(\mathrm{MU}^*(K/T)\big)/I^2.
$$
Then he remarks that the same seems to be true without ``$/I^2$''.

In the present article we consider the following analogue of the above statements for the Morava K-theory.
\begin{Conj}
\label{YC}
Let $K$ be a compact connected Lie group, $T$ its maximal torus, and $\pi\colon K\rightarrow K/T$ the natural projection. Denote by $G=K_{\mathbb C}$ the corresponding (split) reductive group over $\mathbb C$. Then for $A=\mathrm K(n)$ or $A=\mathrm{CK}(n)$ one has
$$
A^*(G)\cong\pi^*\big(A^*_{\mathrm{top}}(K/T)\big).
$$
\end{Conj}

We will prove the above conjecture for groups $K=\mathrm{SO}(m)$ and $K=\mathrm{Spin}(m)$.

\subsection{Iwasawa decomposition} As above, let $K$ be a compact connected Lie group, $T$ its maximal torus, $G=K_{\mathbb C}$ the corresponding reductive group with split maximal torus $T_{\mathbb C}$, and $B$ a Borel subgroup of $G$ containing $T_{\mathbb C}$. 

Using the Iwasawa decomposition $G=KAN$, see~\cite[Theorem~26.3]{iwasawa}, where $A\cong(\mathbb R_{>0})^{l}$, and $N$ is a unipotent radical of $B$ (in particular $A$ and $N$ are contractible), we conclude that $E^*(G)=E^*(K)$ for any (topological) cohomology theory $E^*$. Since $B\cap K=T$, we also conclude that
$$
G/B\cong K/T.
$$
The Bruhat decomposition gives a cellular decomposition for $G/B$, therefore the natural map $A^*(G/B)\rightarrow A^*_{\mathrm{top}}(K/T)$ is an isomorphism for $A=\mathrm K(n),\,\mathrm{CK}(n)$ by Proposition~\ref{HK}.

We can summarize the above discussion with the following result.

\begin{prop}
\label{yagita-restatement}
In the above notation, let $A=\mathrm K(n)$ or $A=\mathrm{CK}(n)$. Let $\pi\colon K\rightarrow K/T$ denote the natural projection, and consider the diagram
$$\xymatrix{
A^*(G/B)\ar@{->>}[r]^{(\pi_{\mathbb C})^{A}}\ar@{=}[d]&A^*(G)\ar[d]\\
A^*_{\mathrm{top}}(K/T)\ar[r]^{\pi^*}&A^*_{\mathrm{top}}(K).
}$$
Then $\pi^*\big(A^*_{\mathrm{top}}(K/T)\big)$ is isomorphic to the image of $A^*(G)$ in $A^*_{\mathrm{top}}(K)$.
\end{prop}

\begin{cl}
\label{yagita-restatement-cl}
Conjecture~\ref{YC} is equivalent to the injectivity of the natural map $A^*(G)\rightarrow A^*_{\mathrm{top}}(K)$.
\end{cl}

\subsection{Connective Morava K-theory of orthogonal groups}

In~\cite[Theorem~6.13]{LPSS} we computed the algebraic {\it periodic} Morava K-theory of special orthogonal groups $\mathrm K(n)^*(\mathrm{SO}_m)$. We can also compute the {\it connective} Morava K-theory of $\mathrm{SO}_m$ as a simple corollary of the techniques of~\cite{LPSS}.

Recall that
$$
\mathrm{Ch}(\mathrm{SO}_m)=\mathbb F_2[e_1,e_2,\ldots,e_s]/(e_i^2=e_{2i})
$$
for $s=\lfloor\frac{m-1}{2}\rfloor$, and $e_k=0$ for $k>s$. We will also need the following results.
\begin{lm}
\label{recall-lpss}
Let $m\in\mathbb N\setminus0$ and let $Q$ denote a split projective quadric of dimension $m-2$. Denote by $l\in\mathrm{Ch}(Q)$ the class of the maximal isotropic subspace in $Q$. Then the following results hold.
\begin{enumerate}
\item
The pullback
$$
\mathrm{Ch}(Q)\rightarrow\mathrm{Ch}(\mathrm{SO}_m)
$$
along the canonical map $\mathrm{SO}_m\rightarrow Q$ sends $l$ to $e_s$.
\item
Let $n\in\mathbb N\setminus0$ be such that $m\geq2^{n+1}+1$. Then the canonical map
$$
\mathrm{CK}(n)(Q)\rightarrow\mathrm{CK}(n)(\mathrm{SO}_m)
$$
sends $v_nl$ and $l^2$ to $0$.
\item
For any $n\in\mathbb N\setminus0$ let $x$ denote the image of $l$ under the canonical map
$$
\mathrm{CK}(n)(Q)\rightarrow\mathrm{CK}(n)(\mathrm{SO}_m).
$$
Then $\mathrm{CK}(n)(\mathrm{SO}_m)/x\cong\mathrm{CK}(n)(\mathrm{SO}_{m-2}).$
\end{enumerate}
\end{lm}
\begin{proof}
For~(1) see~\cite[Lemma~7.2]{PShopf}. The proof of~\cite[Theorem~5.1]{LPSS} can be repeated verbatim to obtain~(2). For~(3) see~\cite[Proposition~5.2]{LPSS}.
\end{proof}

Next we prove the following theorem.

\begin{tm}
\label{ck-algebra}
For all $n$, $m\in\mathbb N\setminus0$ the ring $\mathrm{CK}(n)(\mathrm{SO}_m)$ is {\rm(}non-canonically{\rm)} isomorphic to
$$
R:=\mathbb F_2[v_n][e_1,e_2,\ldots,e_s]/(e_i^2=e_{2i}\ \forall i, v_ne_i=0\text{ for }i\geq2^n),
$$
where $s=\lfloor\frac{m-1}{2}\rfloor$, and $e_k=0$ for $k>s$.
\end{tm}
\begin{proof}
For $m\leq2^{n+1}$ see~\cite[Theorem~6.10]{LPSS}. For $m\geq2^{n+1}+1$ we will argue by induction.

In the notation of Lemma~\ref{recall-lpss} consider arbitrary lifts $\widetilde e_i$ of the elements ${e_i\in\mathrm{CK}(n)(\mathrm{SO}_{m-2})}$ to $\mathrm{CK}(n)(\mathrm{SO}_m)$. These lifts together with $x$ generate $\mathrm{CK}(n)(\mathrm{SO}_m)$ as an algebra.

First we will construct a map from $R$ to $\mathrm{CK}(n)(\mathrm{SO}_m)$ sending $e_i$ to $\widetilde e_i$ for $i<s$, and sending $e_s$ to $x$.

Using Lemma~\ref{recall-lpss}\,(2) we have $v_nx=0$ and $x^2=0$. 

We claim that $v_n\widetilde e_i=0$ for $2^n\leq i<s$. Using the inductive assumption we obtain $v_n\widetilde e_i=xz$ for some $z\in\mathrm{CK}(n)(\mathrm{SO}_m)$. Decompose $z$ as a linear combination of monomials in $\widetilde e_i$ and $x$. Since $v_nx=0$, we can assume that this linear combination has $\mathbb F_2$-coefficients, rather than just $\mathbb F_2[v_n]$-. Moreover, since $x^2=0$ we can assume that this linear combination consists only of monomials in $\widetilde e_i$ (and not in $x$). However, the image of $xz$ in $\mathrm{Ch}(\mathrm{SO}_m)$ is $0$. In other words, the image of $z$ in $\mathrm{Ch}(\mathrm{SO}_m)$ is the linear combination in $e_1,\ldots,e_{s-1}$ which is annihilated by $e_s$. This can only happen if the image of $z$ in $\mathrm{Ch}(\mathrm{SO}_m)$ is zero. In other words, $z$ is divisible by $v_n$, but this implies that $xz=0$. 

Similarly, one shows that $\widetilde e_i^2-\widetilde e_{2i}=0$ for $2i<s$. Finally, we have to show that $\widetilde e_{s/2}^{\,2}=x$ for $s$ even. By induction we know that $\widetilde e_{s/2}^{\,2}=xz$ for some $z$. Then $\widetilde e_{s/2}^{\,2}-x=x(z-1)$ and we can argue for $z-1$ as above to obtain the claim.

In other words, the map from $R$ to $\mathrm{CK}(n)(\mathrm{SO}_m)$ is well-defined (and surjective). Then $\mathrm{CK}(n)(\mathrm{SO}_m)=R/(f_i)$ for some $f_1,\ldots,f_k\in R$. However, the images of $f_i$ in $\mathrm{CK}(n)(\mathrm{SO}_{m-2})$ should also be $0$, therefore $f_i=e_sg_i$ for some $g_i$ by induction. We can again assume that $g_i$ are polynomials in $e_i$ for $i<s$ with $\mathbb F_2$-coefficients. However, $f_i$ should equal $0$ modulo $v_n$, since we can specialize $\mathrm{CK}(n)(\mathrm{SO}_m)$ to $\mathrm{Ch}(\mathrm{SO}_m)$. This again implies that $g_i$ are divisible by $v_n$, and therefore $f_i=0$.
\end{proof}

\subsection{Injectivity Theorem}

In this subsection we will prove Conjecture~\ref{YC} for special othogonal and spinor groups. 

\begin{lm}
\label{ahss-injective}
Assume that for a compact connected Lie group $K$ the Atiyah--Hirzebruch spectral sequence $\mathrm H^p\big(K;\,\mathrm K(n)_q^{\mathrm{top}}(*)\big)\Rightarrow \mathrm K(n)_{\mathrm{top}}^{p+q}(K)$ collapses at the second page. 
Then the natural maps 
$$
\mathrm{CK}(n)^*(K_{\mathbb C})\rightarrow\mathrm{CK}(n)_{\mathrm{top}}^*(K)\ \text{ and }\ \mathrm{K}(n)^*(K_{\mathbb C})\rightarrow\mathrm{K}(n)_{\mathrm{top}}^*(K)
$$
are injective.
\end{lm}
\begin{proof}
Let $G=K_{\mathbb C}$ denote the corresponding reductive group. We conclude by Proposition~\ref{ahss-cohomological} that $\mathrm{CK}(n)_{\mathrm{top}}^*(K)$ is a free $\mathbb F_2[v_n]$-module. Consider the diagram
$$\xymatrix{
\mathrm{CK}(n)^*(G)\ar[r]^{}\ar[d]^{}&\mathrm{CK}(n)_{\mathrm{top}}^*(K)\ar[d]^{}\\
\mathrm{Ch}^*(G)\,\ar@{>->}[r]&\mathrm H^*(K;\,\mathbb F_2).
}$$
Take an element $x\in\mathrm{CK}(n)^*(G)$ which is not divisible by $v_n$. Then the injectivity of the bottom horizontal arrow implies that the image of $x$ in $\mathrm{CK}(n)_{\mathrm{top}}^*(K)$ cannot be zero. Since $\mathrm{CK}(n)_{\mathrm{top}}^*(K)$ cannot have a $v_n$-torsion, we conclude that the top horizontal arrow is injective. 

Finally, since $\mathrm{CK}(n)_{\mathrm{top}}^*(K)$ cannot have a $v_n$-torsion, and the natural map $$\mathrm{CK}(n)_{\mathrm{top}}^*(K)\rightarrow\mathrm{K}(n)_{\mathrm{top}}^*(K)$$ coincides with the localization, we conclude that this map is injective. This implies that 
$$
\mathrm{K}(n)^*(G)\rightarrow\mathrm{K}(n)_{\mathrm{top}}^*(K)
$$
is also injective.
\end{proof}

\begin{rk}
In this paper we are only interested in orthogonal groups, and therefore only work with Morava K-theories for $p=2$. However, we remark that the above argument clearly remains valid for odd $p$. 
\end{rk}

\begin{tm}
\label{injectivity-theorem}
Let $K=\mathrm{SO}(m)$ and $K=\mathrm{Spin}(m)$ {\rm(}for any $m\in\mathbb N${\rm)}, and $G=K_{\mathbb C}$. Then the natural maps 
$$
\mathrm{K}(n)^*(G)\rightarrow\mathrm{K}(n)_{\mathrm{top}}^*(K)\ \text{ and }\ \mathrm{CK}(n)^*(G)\rightarrow\mathrm{CK}(n)_{\mathrm{top}}^*(K)
$$
are injective.
\end{tm}
\begin{proof}
First, assume that $m\leq2^{n+1}$. Under this condition, the Atiyah--Hirzebruch spectral sequence $\mathrm H^p\big(K;\,\mathrm K(n)_q^{\mathrm{top}}(*)\big)\Rightarrow \mathrm K(n)_{\mathrm{top}}^{p+q}(K)$ collapses by~\cite[Theorems~2.4 and~3.2]{Nis}. 
Therefore the claim follows from Lemma~\ref{ahss-injective}.

Next, consider the case $m>2^{n+1}$, and assume that the natural map from the algebraic Morava K-theory $\mathrm K(n)^*(G)$ to the topological one $\mathrm K(n)^*_{\mathrm{top}}(K)$ is not injective. Consider the natural inclusion of $K_0=\mathrm{SO}(2^{n+1})$ for $m$ even or $K_0=\mathrm{SO}(2^{n+1}-1)$ for $m$ odd into $K$, and denote $G_0=(K_0)_{\mathbb C}$. Consider the following diagram
$$
\xymatrix{
\mathrm K(n)^*(G)\ar[d]^{\cong}\ar[r]&\mathrm K(n)^*_{\mathrm{top}}(K)\ar[d]\\
\mathrm K(n)^*(G_0)\ar@{^(->}[r]&\mathrm K(n)^*_{\mathrm{top}}(K_0)
}
$$
where the left vertical arrow is an isomorphism by~\cite[Theorem~5.1]{LPSS}. Then a simple diagram chase provides a contradiction.

Finally, assume that the natural map from $\mathrm{CK}(n)^*(G)$ to $\mathrm{CK}(n)^*_{\mathrm{top}}(K)$ is not injective, and take an element $x$ in the kernel. Since $x$ maps to $0$ in $\mathrm{Ch}^*(G)$, we conclude that $x=v_ny$ for some $y$. On the other hand, since the image of $x$ in $\mathrm{K}(n)^*(G)$ is also $0$, we conclude that $x$ (and therefore $y$) is a $v_n^{\mathbb Z}$-torsion. However, using Theorem~\ref{ck-algebra} we conclude that if $y$ is a $v_n^{\mathbb Z}$-torsion, then in fact $v_ny=0$. This finishes the proof.
\end{proof}

Conjecture~\ref{YC} for special othogonal and spinor groups now follows from Corollary~\ref{yagita-restatement-cl}.

\begin{rk}
Observe that in the case $m>2^{n+1}$ the proof of Theorem~\ref{injectivity-theorem} relies on the results of~\cite{LPSS}. We do not know how to prove Conjecture~\ref{YC} without actually computing $\mathrm K(n)^*(G)$ (as a module).
\end{rk}

\section{Computation of the co-multiplication}
\label{topology}

\subsection{Statement of results}

In~\cite[Theorem~6.13]{LPSS} the algebra structure of $\mathrm K(n)^*(\mathrm{SO}_m)$ is described. As an application of the Yagita conjecture, we can also deduce the co-algebra structure using results in topology.

\begin{tm}
\label{co-mult}
The algebra structure of $\mathrm K(n)^*(\mathrm{SO}_m)$ is given by
$$
\mathbb F_2[v_n^{\pm1}][e_1,e_2,\ldots,e_s]/(e_i^2-e_{2i})
$$
where $s=\mathrm{min}\left(\lfloor\frac{m-1}{2}\rfloor,\,2^{n}-1\right)$ and $e_{2i}$ stands for $0$ if $2i>s$. The reduced co-multiplication $\widetilde\Delta(x)=\Delta(x)-x\otimes1-1\otimes x$ is given by
$$
\widetilde\Delta(e_{\langle 2k\rangle})=\sum_{i=0}^{\nu_2(k)}v_n^{i+1}\,e_{\langle k/2^{i}\rangle}\otimes e_{\langle k/2^{i}\rangle}\,\prod_{j=0}^{i-1}\left(e_{\langle k/2^{j}\rangle}\otimes1+1\otimes e_{\langle k/2^{j}\rangle}\right)
$$
where $\langle t\rangle$ stands for $2^n-1-t$, $k>0$ and $\nu_2(k)$ is a $2$-adic valuation of $k$, i.e., $k/2^{\,\nu_2(k)}$ is an odd integer, and
$$
\widetilde\Delta(e_{2^n-1})=v_n\,e_{2^{n}-1}\otimes e_{2^{n}-1}.
$$
Moreover, $e_1$ coincides with the first Chern class $c_1^{\mathrm{K}(n)}(e_1^{\mathrm{CH}})$ of the generator $e_1^{\mathrm{CH}}$ of $\mathrm{CH}^1(\mathrm{SO}_m;\mathbb Z)=\mathrm{Pic}(\mathrm{SO}_m)$.
\end{tm}
For the computation of $\mathrm{CH}^1(\mathrm{SO}_m;\mathbb Z)$ see~\cite[Th\'eor\`eme~7]{marlin}, and $c_1^{\mathrm{K}(n)}$ is defined in~\cite[Chapter~I, Subsection~1.1]{LM}).

 We remark, that as a corollary, we also obtain the Hopf algebra structure of 
$$
\mathrm K(n)^*(\mathrm{Spin}_m)=\mathrm K(n)^*(\mathrm{SO}_m)/e_1,
$$
cf.~\cite[Corollary~6.15]{LPSS}. Below we provide a few examples.
\begin{Ex}
\label{primitive}
It is easy to see that all $e_{2j-1}^2=e_{4j-2}\in \mathrm{K}(n)^*(\mathrm{SO}_m)$ are primitive. Indeed, let $2j-1=\langle 2k\rangle=2^n-1-2k$ for some $0\leq k\leq 2^{n-1}-1$. Then $2\langle k/2^i\rangle\geq2^{n}$ for all $i\geq0$, and therefore
$$
e_{\langle k/2^i\rangle}^2=e_{2\langle k/2^i\rangle}=0.
$$
Similarly, for $m\leq2^n$ all $e_{2i-1}$ are primitive.
\end{Ex}
\begin{Ex}
Since $e_1^{2^n}=0$ in $\mathrm K(n)^*(\mathrm{SO}_m)$, we conclude that
$$
\Delta(e_{1})=e_1\otimes1+1\otimes e_1+v_n\,e_1^{2^{n-1}}\otimes e_1^{2^{n-1}}=\mathrm{FGL}_{\mathrm K(n)}(e_1\otimes1,\,1\otimes e_1),
$$
by, e.g.,~\cite[Subsection~2.3]{PSlog} (cf. also~\cite[Lemma~6.12]{LPSS}). Similarly, for odd $j$ we have
$$
\Delta(e_{2j-1})=e_{2j-1}\otimes1+1\otimes e_{2j-1}+v_n\,e_{2^{n-1}-1+j}\otimes e_{2^{n-1}-1+j}.
$$
\end{Ex}

\subsection{Homology of special orthogonal groups}
\label{top-morava}

We remark that the topological Morava K-theory of the special orthogonal group is computed as an {\it $\mathbb F_2[v_n^{\pm1}]$-module} in~\cite[Theorem~1]{R90} and~\cite[Theorem~2.10]{Nis}. The algebra and co-algebra structure of $\mathrm K(n)^*_{\mathrm{top}}\big(\mathrm{SO}(m)\big)$ are not known for an arbitrary $m$, see~\cite{R08,R12} for partial results. 

However, for $m=2^{n+1}-1$ both algebra and co-algebra structure of the topological Morava K-theory $\mathrm K(n)^*_{\mathrm{top}}\big(\mathrm{SO}(2^{n+1}-1)\big)$ are known by~\cite{R97}. Due to the stabilization of the {\it algebraic} Morava K-theory~\cite[Theorem~5.1]{LPSS}, this is already sufficient for the computation of the co-multiplication of $\mathrm K(n)^*\big(\mathrm{SO}_m\big)$ for all $m$.

First, we recall the Hopf algebra structure of $\mathrm H_*\big(\mathrm{SO}(2^{n+1}-1);\,\mathbb F_2\big)$. 

For a Hopf algebra $H=\mathbb F_2[x]/x^{2^k}$ with $x$ primitive we denote the dual basis of the dual Hopf algebra $H^\vee$ by $\gamma_j$, i.e., $\langle x^i,\,\gamma_j\rangle=\delta_{ij}$. One easily checks that the multiplication and the co-multiplication in $H^\vee$ are given by 
$$
\gamma_i\cdot\gamma_j=\frac{(i+j)!}{i!j!}\gamma_{i+j}\  \text{ and }\ \Delta(\gamma_k)=\sum_{t=0}^k\gamma_t\otimes\gamma_{k-t}.
$$ 
Such an algebra $H^\vee$ is called the algebra of divided powers, and we denote it by $\Gamma_k$.
By~\cite[Theorem~6\,(d)]{Kac} we know that
$$
\mathrm H^*\big(\mathrm{SO}(2^{n+1}-1);\,\mathbb F_2\big)=\Lambda_2(x_{2^n+1},\ldots,x_{2^{n+1}-3})\otimes\mathbb F_2[x_1,x_3,\ldots,x_{2^n-1}]\left/\left(x_{2i-1}^{2^{(k_i+1)}}\right)\right.\!,
$$
where 
\begin{align}
\label{ki}
k_i=\left\lfloor\mathrm{log}_2\left(\frac{2^{n+1}-2}{2i-1}\right)\right\rfloor
\end{align}
for $i\leq2^{n-1}$. Set additionally $k_i=0$ for $i>2^{n-1}$ (in fact, formula~(\ref{ki}) also gives $k_i=0$ for $i>2^{n-1}$). Then
$$
\mathrm H^*\big(\mathrm{SO}(2^{n+1}-1);\,\mathbb F_2\big)\cong\bigotimes_{i=1}^{2^n-1}\mathbb F_2[x_{2i-1}]/x_{2i-1}^{2^{(k_i+1)}},
$$
and dualizing the above formula we obtain
\begin{equation}
\label{homology-so}
\mathrm H_*\big(\mathrm{SO}(2^{n+1}-1);\,\,\mathbb F_2\big)\cong\bigotimes_{i=1}^{2^n-1}\Gamma_{k_i+1}(\beta_{2i-1}),
\end{equation}
where $\Gamma_{k_i+1}(\beta_{2i-1})$ denotes the dual Hopf algebra to $\mathbb F_2[x_{2i-1}]/x_{2i-1}^{2^{(k_i+1)}}$, and the basis elements of $\Gamma_{k_i+1}(\beta_{2i-1})$ are denoted by $\gamma_t(\beta_{2i-1})$. We write $\beta_{2i-1}$ for $\gamma_1(\beta_{2i-1})$.

\subsection{Rao's Theorem}

In fact, the multiplication and the co-multiplication of 
$$
\mathrm P(n)_*^{\mathrm{top}}\big(\mathrm{SO}(2^{n+1}-1)\big)
$$ 
are computed in~\cite{R97}. We will reproduce this result in this subsection. 

Observe also that $\mathrm P(n)^*_{\mathrm{top}}\big(\mathrm{SO}(2^{n+1}-1)\big)$ admits a canonical co-algebra structure induced by the multiplication in $\mathrm{SO}(2^{n+1}-1)$, however, it is {\it not} a Hopf algebra. More precisely, the co-multiplication map $\Delta$ does not preserve the multiplication, $\Delta(xy)\neq\Delta(x)\Delta(y)$. The reason for this is that the multiplicative structure on $\mathrm P(n)^{\mathrm{top}}$ does not commute with a twist $\tau\colon X\wedge Y\rightarrow Y\wedge X$ (see~\cite[Corollary~1.6]{R93}).

It is proven in~\cite[Theorem~1.1]{R89} that there exists a filtration on the topological homology theory $\mathrm P(n)_*^{\mathrm{top}}\big(\mathrm{SO}(2^{n+1}-1)\big)$ (which respects the multiplication and the co-multiplication) such that the graded (co-)algebra associated with this filtration is isomorphic to $\mathrm H_*\big(\mathrm{SO}(2^{n+1}-1);\,\,\mathbb F_2[v_n,v_{n+1},\ldots]\big)$.

Moreover, in~\cite{R97} the lifts of $\gamma_t(\beta_{2i-1})$ from~(\ref{homology-so}) to $\mathrm P(n)_*^{\mathrm{top}}\big(\mathrm{SO}(2^{n+1}-1)\big)$ (which we will denote by the same symbols) are chosen in such a way, that the following theorem holds.

\begin{tm}[Rao]
\label{rao}
For $1\leq i\leq2^n-1$ and $1\leq t\leq 2^{k_i}$, where $k_i$ are given by~{\rm(}\ref{ki}{\rm)}, there exist elements $\gamma_t(\beta_{2i-1})\in\mathrm P(n)_*^{\mathrm{top}}\big(\mathrm{SO}(2^{n+1}-1)\big)$ which lift the corresponding elements of $\mathrm H_*(\mathrm{SO}(2^{n+1}-1);\,\mathbb F_2)$ described in Subsection~\ref{top-morava}. Moreover the following statements hold.
\begin{enumerate}
\item
The ordered set
$$
M=\{\gamma_{2^s}(\beta_{2i-1})\mid1\leq i\leq2^n-1,\ 0\leq s\leq k_i\}
$$
{\rm(}with an arbitrary choice of order{\rm)} simply generates $\mathrm P(n)_*^{\mathrm{top}}\big(\mathrm{SO}(2^{n+1}-1)\big)$. 
By this we mean that the monomials $g_1g_2\ldots g_k$, for $g_i<g_{i+1}\in M$, $k\geq0$, form a basis of $\mathrm P(n)_*^{\mathrm{top}}\big(\mathrm{SO}(2^{n+1}-1)\big)$ over $\F_2[v_n,v_{n+1},\ldots]$. 
\item
For $1\leq i\leq2^{n}-1$, $1\leq t\leq2^{k_i}$ we have
$$
\Delta(\gamma_t(\beta_{2i-1}))=\sum_{s=0}^t\gamma_s(\beta_{2i-1})\otimes\gamma_{t-s}(\beta_{2i-1}).
$$
\item
For $1\leq i\leq2^{n-1}$, $1\leq s,t< 2^{k_i}$ we have
$$
\gamma_s(\beta_{2i-1})\cdot\gamma_t(\beta_{2i-1})=\frac{(s+t)!}{s!\,t!}\gamma_{s+t}(\beta_{2i-1}).
$$
\item
For $1\leq i\leq 2^{n-1}$ and $j=(2i-1)2^{k_i-1}-2^{n-1}$ we have
$$
\gamma_{2^{k_i}}(\beta_{2i-1})^2=[\gamma_{2^{k_i}}(\beta_{2i-1}),\,\beta_{2i-1}]\cdot\gamma_{2^{k_i}-1}(\beta_{2i-1})+v_n\gamma_2(\beta_{2j-1}).
$$
\item
The quotient algebra of $\mathrm P(n)_*^{\mathrm{top}}\big(\mathrm{SO}(2^{n+1}-1)\big)$ modulo the ideal generated by $\beta_{2i-1}$, $1\leq i\leq2^n-1$, is commutative.
\end{enumerate}
\end{tm}
\begin{proof}
For~(1) see~\cite[Proposition~3.1]{R97}, for~(2) see~\cite[Proposition~2.4 and Lemma~3.4]{R97}, for~(3) see~\cite[Lemma~3.4]{R97}, for~(4) see~\cite[Proposition~5.8]{R97}. To obtain~(5) it suffices to show that the commutators of all generators are $0$ modulo $\beta_{2i-1}$. This is shown in~\cite[Lemmas~3.2--3.4, Corollary~5.3, Propositions~5.4--5.5, Lemma~5.7]{R97}.
\end{proof}

\begin{rk}
In fact, in~\cite{R97} the rest of the algebra and co-algebra structures of $\mathrm P(n)_*^{\mathrm{top}}\big(\mathrm{SO}(2^{n+1}-1)\big)$ is also computed, but we will not need it in this paper.
\end{rk}

It will be more convenient for us to work with $\mathrm{CK}(n)_{\mathrm{top}}$ rather than $\mathrm P(n)_{\mathrm{top}}$. Recall that 
$$
\mathrm{CK}(n)_*^{\mathrm{top}}(\mathrm{SO}(2^{n+1}-1))\cong\mathrm P(n)_*^{\mathrm{top}}(\mathrm{SO}(2^{n+1}-1))\otimes_{\mathbb F_2[v_n,v_{n+1},\ldots]}\mathbb F_2[v_n],
$$
and
$$
\mathrm{CK}(n)^*_{\mathrm{top}}(\mathrm{SO}(2^{n+1}-1))\cong\mathrm{CK}(n)_*^{\mathrm{top}}(\mathrm{SO}(2^{n+1}-1))^\vee
$$
by Proposition~\ref{ahss-homological}.

\subsection{Algebraic part of $\mathrm{CK}(n)_*^{\mathrm{top}}\big(\mathrm{SO}(2^{n+1}-1)\big)$}

We recall that for a field extension 
$k\subseteq L$,
the restriction map 
$\mathrm{res}^L_k$ induces an isomorphism on $A^*(\mathrm{SO}_m)$ for any free theory $A^*$. Therefore, instead of an arbitrary field of characteristic zero we can take $k=\mathbb Q$ or $k=\mathbb C$. The latter choice allows us to apply Yagita's conjecture.

Recall that the natural map from $\mathrm{Ch}^*(\mathrm{SO}_{m})$ to $\mathrm{H}^*(\mathrm{SO}(m);\,\mathbb F_2)$ is injective (see Subsection~\ref{yagita-conjecture}), and its image coincides with the set of squares in $\mathrm{H}^*(\mathrm{SO}(m);\,\mathbb F_2)$ (cf.~\cite[Remark after Theorem~6]{Kac}). In Proposition~\ref{quotient-ch} below we will dualize the above statement to determine the ``algebraic part'' of $\mathrm H_*(\mathrm{SO}_m;\,\mathbb F_2)$. 

\begin{prop}
\label{quotient-ch}
In the notation of Subsection~\ref{top-morava}, $\mathrm{Ch}^*(\mathrm{SO}_{2^{n+1}-1})^\vee$ is isomorphic as a Hopf algebra to the quotient of $\mathrm{H}_*(\mathrm{SO}(2^{n+1}-1);\,\mathbb F_2)$ modulo the ideal generated by $\beta_{2i-1}$ for all $1\leq i\leq 2^n-1$.
\end{prop}
We recall the proof in the Appendix~\ref{app-alg-part} for the convenience of the reader. 

\begin{cl}
\label{cor-ch}
In the notation of Proposition~\ref{quotient-ch}, denote by $\gamma_{t}(\alpha_{2i-1})$ the images of $\gamma_{2t}(\beta_{2i-1})$ in $\mathrm{Ch}^*(\mathrm{SO}_{2^{n+1}-1})^\vee$ for $1\leq i\leq2^{n-1}$, $1\leq t<2^{k_i}$. Then 
$$
\mathrm{Ch}^*(\mathrm{SO}_{2^{n+1}-1})^\vee\cong\bigotimes_{i=1}^{2^{n-1}}\Gamma_{k_i}(\alpha_{2i-1})
$$ 
as a Hopf algebra.
\end{cl}

Now we will prove analogues of these results for the connective Morava K-theory. As above, let $\mathrm{CK}(n)^*(\mathrm{SO}_m)^\vee$ denote the Hopf algebra over $\mathbb F_2[v_n]$ dual to $\mathrm{CK}(n)^*(\mathrm{SO}_m)$. Recall that we denote $\gamma_t(\beta_{2i-1})$ in homology and their lifts to $\mathrm{CK}(n)^{\mathrm{top}}_*(\mathrm{SO}_m)$ by the same symbols by some abuse of notation.
\begin{prop}
\label{quotient}
In the notation of Theorem~\ref{rao}, $\mathrm{CK}(n)^*(\mathrm{SO}_{2^{n+1}-1})^\vee$ is isomorphic to the quotient of $\mathrm{CK}(n)_*^{\mathrm{top}}(\mathrm{SO}(2^{n+1}-1))$ modulo the ideal generated by $\beta_{2i-1}$ for all $1\leq i\leq 2^n-1$ as an algebra and a co-algebra.
\end{prop}
\begin{proof}
Consider the map
\begin{align}
\label{themap}
\xymatrix{
\mathrm{CK}(n)^*(\mathrm{SO}_{2^{n+1}-1})^\vee&&\mathrm{CK}(n)^{\mathrm{top}}_*(\mathrm{SO}({2^{n+1}-1}))\ar@{->>}[ll]_{\varphi}
}
\end{align}
dual to the natural inclusion $\mathrm{CK}(n)^*(\mathrm{SO}_{2^{n+1}-1})\hookrightarrow\mathrm{CK}(n)_{\mathrm{top}}^*(\mathrm{SO}({2^{n+1}-1}))$, and denote 
$$
\overline\gamma_t(\beta_{2i-1})=\varphi\big(\gamma_t(\beta_{2i-1})\big).
$$
Using Propositions~\ref{ahss-homological} and~\ref{quotient-ch} we conclude that $\mathrm{CK}(n)^*(\mathrm{SO}_{2^{n+1}-1})^\vee$ is generated over $\mathbb F_2[v_n]$ by the monomials $g_1\ldots g_k$, where $g_i$ lie in 
$$
\{\overline\gamma_{2^s}(\beta_{2i-1})\mid1\leq i\leq2^{n-1},\ 1\leq s\leq k_i\}.
$$
In fact, comparing the ranks, we conclude that these monomials form a basis. Moreover, by Proposition~\ref{quotient-ch} 
we conclude that $\overline\beta_{2i-1}$ are divisible by $v_n$. Let 
$$
\overline\beta_{2i-1}=v_n^aP+v_n^{a+1}R
$$
in $\mathrm{CK}(n)^*(\mathrm{SO}_{2^{n+1}-1})^\vee$, where $P$ is not divisible by $v_n$, i.e., the image of $P$ in $\mathrm{Ch}^*(\mathrm{SO}_{2^{n+1}-1})^\vee$ is non-zero. Since~(\ref{themap}) is a morphism of algebras and co-algebras, and $\beta_{2i-1}$ is primitive by~Theorem~\ref{rao}, we conclude that $P$ is primitive modulo $v_n$. This implies that $P$ coincides with a sum of $\overline\gamma_2(\beta_{2j-1})$ for some $1\leq j\leq 2^{n-1}$ modulo $v_n$ by Corollary~\ref{cor-ch}. Comparing the gradings 
$$
2i-1=2(2j-1)-a(2^n-1)
$$
we conclude that $a=1$ and $P$ is congruent just to $\overline\gamma_2(\beta_{2j-1})$ modulo $v_n$ (rather than to a sum of several $\overline\gamma_2(\beta_{2j-1})$). Moreover $2j-1\geq2^{n-1}-1$ (under this condition $k_j=1$), $i\leq j$ (in particular, $i\leq2^{n-1}$), and $i=j$ implies that $i=2^{n-1}$. 

Take the largest $i$ such that $\overline\beta_{2i-1}\neq0$ and let
$$
\overline\beta_{2i-1}=v_n\overline\gamma_2(\beta_{2j-1})+v_n^{2}R.
$$
Since $k_j=1$ we conclude that 
$$
\gamma_2(\beta_{2j-1})^2=[\gamma_2(\beta_{2j-1}),\,\beta_{2j-1}]\cdot\beta_{2j-1}+v_n\gamma_2(\beta_{2i-1})
$$
by Theorem~\ref{rao}\,(4). On the other hand,
$$
\beta_{2i-1}^2=0
$$
by Theorem~\ref{rao}\,(3). We do not know yet that $\mathrm{CK}(n)^*(\mathrm{SO}_{2^{n+1}-1})^\vee$ is commutative, however, all commutators are divisible by $v_n$ in $\mathrm{CK}(n)^*(\mathrm{SO}_{2^{n+1}-1})^\vee$, and we obtain
$$
0=\overline\beta_{2i-1}^{\,2}=v_n^2\,\overline\gamma_2(\beta_{2j-1})^2+v_n^4R'.
$$
For $i<2^{n-1}$ we have $\overline\beta_{2j-1}=0$ (by the choice of $i$), and therefore 
$$
0=v_n^3\,\overline\gamma_2(\beta_{2i-1})+v_n^4R',
$$
but the image of $\gamma_2(\beta_{2i-1})$ in $\mathrm{Ch}^*(\mathrm{SO}_{2^{n+1}-1})^\vee$ is non-zero. This provides a contradiction.

Finally, let $i=2^{n-1}$. Then
\begin{multline*}
0=\overline\beta_{2^n-1}^{\,2}=v_n^2\,\overline\gamma_2(\beta_{2^n-1})^2+v_n^4R'=\\
=v_n^3\,\overline\gamma_2(\beta_{2^n-1})+v_n^2\,[\overline\gamma_2(\beta_{2^n-1}),\,\overline\beta_{2^n-1}]\cdot\overline\beta_{2^n-1}+v_n^4R',
\end{multline*}
and using again the fact that all commutators are divisible by $v_n$ and $\overline\beta_{2^n-1}=v_n\overline\gamma_2(\beta_{2^n-1})+v_n^2R$, we obtain
$$
0=v_n^3\,\overline\gamma_2(\beta_{2^n-1})+v_n^4R'',
$$
which again provides a contradiction.

Therefore all $\overline\beta_{2i-1}$ are equal to zero, and $\varphi$ factors through the quotient modulo the ideal generated by $\beta_{2i-1}$. Comparing the ranks we obtain the claim.
\end{proof}

\begin{cl}
\label{commutative}
The Hopf algebra $\mathrm{CK}(n)^*(\mathrm{SO}_{2^{n+1}-1})^\vee$ is commutative.
\end{cl}
\begin{proof}
This follows from Theorem~\ref{rao}\,(5).
\end{proof}

We are going to reuse the notation for $\gamma_t(\alpha_{2i-1})$ in $\mathrm{Ch}^*(\mathrm{SO}_{2^{n+1}-1})^\vee$ also for their lifts to $\mathrm{CK}(n)^*(\mathrm{SO}_{2^{n+1}-1})^\vee$. In the notation of the proof of Proposition~\ref{quotient}, denote 
$$
\gamma_{t}(\alpha_{2i-1})=\overline\gamma_{2t}(\beta_{2i-1})
$$ 
in $\mathrm{CK}(n)^*(\mathrm{SO}_m)^\vee$ for $1\leq i\leq2^{n-1}$, $1\leq t\leq2^{k_i-1}$, and denote additionally
$$
\gamma_{2^{k_i-1}+t}(\alpha_i)=\gamma_{2^{k_i-1}}(\alpha_i)\gamma_{t}(\alpha_i)
$$
for $t<2^{k_i-1}$. We also reuse below the notation $\Gamma_{k_i}(\alpha_{2i-1})$ for the algebra of divided differences over $\mathbb F_2[v_n]$ (rather than $\mathbb F_2$).

\begin{prop}
\label{answer-dual}
In the notation above, 
$$
\mathrm{CK}(n)^*(\mathrm{SO}_{2^{n+1}-1})^\vee\cong\bigotimes_{i=1}^{2^{n-1}}\Gamma_{k_i}(\alpha_{2i-1})
$$ 
is an isomorphism of co-algebras, and the multiplication table of $\mathrm{CK}(n)^*(\mathrm{SO}_{2^{n+1}-1})^\vee$ can be deduced from the identities
\begin{align}
\label{raz-d}
\gamma_s(\alpha_{2i-1})\cdot\gamma_t(\alpha_{2i-1})=\frac{(s+t)!}{s!\,t!}\gamma_{s+t}(\alpha_{2i-1})
\end{align}
for $1\leq i\leq2^{n-1}$, $1\leq s,t< 2^{k_i-1}$, and
\begin{align}
\label{dvas-d}
\gamma_{2^{k_i-1}}(\alpha_{2i-1})^2=v_n\alpha_{2j-1}
\end{align}
for $1\leq i\leq2^{n-1}$, $2j-1=(2i-1)2^{k_i}-2^{n}+1$.
\end{prop}

\begin{proof}
The identity~(\ref{raz-d}) and the formula
\begin{align}
\label{tris}
\Delta(\gamma_t(\alpha_{2i-1}))=\sum_{s=0}^t\gamma_s(\alpha_{2i-1})\otimes\gamma_{t-s}(\alpha_{2i-1})
\end{align}
for $1\leq t\leq 2^{k_i-1}$ follow from Theorem~\ref{rao}. Using this and the fact that $\mathrm{CK}(n)^*(\mathrm{SO}_{2^{n+1}-1})^\vee$ is a Hopf algebra we can prove~(\ref{tris}) for all $1\leq t<2^{k_i}$, and obtain the first claim. Finally,~(\ref{dvas-d}) follows from Theorem~\ref{rao}\,(4) and Corollary~\ref{commutative}.
\end{proof}

\subsection{The co-algebra structure of $\mathrm{CK}(n)^*(\mathrm{SO}_m)$}

Dualizing the above result, we obtain the following 
\begin{prop}
\label{prop-answer}
The algebraic Morava K-theory $\mathrm{CK}(n)^*(\mathrm{SO}_{2^{n+1}-1})$ can be identified as an algebra with 
$$
\mathbb F_2[v_n][e_1,e_3\ldots,e_{2^n-1}]/(e_{2i-1}^{2^{k_i}})
$$
where $k_i$ are given by~{\rm(}\ref{ki}{\rm)}. Denote additionally $e_{2k}=e_k^2$. The reduced co-multiplication $\widetilde\Delta(x)=\Delta(x)-x\otimes1-1\otimes x$ of $\mathrm{CK}(n)^*(\mathrm{SO}_{2^{n+1}-1})$ under the above identification is given by
\begin{align}
\label{hereitis}
\widetilde\Delta(e_{\langle 2k\rangle})=\sum_{i=0}^{\nu_2(k)}v_n^{i+1}\,e_{\langle k/2^{i}\rangle}\otimes e_{\langle k/2^{i}\rangle}\,\prod_{j=0}^{i-1}\left(e_{\langle k/2^{j}\rangle}\otimes1+1\otimes e_{\langle k/2^{j}\rangle}\right)
\end{align}
where $\langle t\rangle$ stands for $2^n-1-t$, $k>0$, and $\nu_2(k)$ is the $2$-adic valuation of $k$, i.e., $k/2^{\,\nu_2(k)}$ is an odd integer, and
$$
\widetilde\Delta(e_{2^n-1})=v_n\,e_{2^{n}-1}\otimes e_{2^{n}-1}.
$$
\end{prop}

\begin{proof}
We will denote by $\Big\langle\ ,\ \Big\rangle$ the natural pairing of $\mathrm{CK}(n)^*(\mathrm{SO}_{2^{n+1}-1})$ with $\mathrm{CK}(n)^*(\mathrm{SO}_{2^{n+1}-1})^\vee$.

Let $e_{2i-1}$ denote the dual elements to $\alpha_{2i-1}$. Then the first claim follows from Proposition~\ref{answer-dual}. To compute $\Delta(e_{2^n-1})$ observe that
$$
\Big\langle\Delta(e_{2^n-1}),\,\gamma_{t_1}(\alpha_1)\ldots\gamma_{t_{2^n-1}}(\alpha_{2^n-1})\otimes\gamma_{s_1}(\alpha_1)\ldots\gamma_{s_{2^n-1}}(\alpha_{2^n-1})\Big\rangle
$$
can only be non-zero for the basis elements $1\otimes\alpha_{2^n-1}$, $\alpha_{2^n-1}\otimes1$ and $\alpha_{2^n-1}\otimes\alpha_{2^n-1}$ in the right hand side of the pairing $\Big\langle\ ,\,\Big\rangle$ by Proposition~\ref{answer-dual}. Similarly, for odd $j$ we see that $2^{n-1}-1+j$ is even, i.e., if $2j-1=(2i-1)2^{k_i}-2^{n}+1$ then $k_i>1$. This also implies that 
$$
\Big\langle\Delta(e_{2j-1}),\,\gamma_{t_1}(\alpha_1)\ldots\gamma_{t_{2^n-1}}(\alpha_{2^n-1})\otimes\gamma_{s_1}(\alpha_1)\ldots\gamma_{s_{2^n-1}}(\alpha_{2^n-1})\Big\rangle
$$
can only be non-zero for the basis elements $1\otimes\alpha_{2j-1}$, $\alpha_{2j-1}\otimes1$ and $\gamma_{2^{k_i-1}}(\alpha_{2i-1})\otimes\gamma_{2^{k_i-1}}(\alpha_{2i-1})$, where  $2j-1=(2i-1)2^{k_i}-2^{n}+1$ by Proposition~\ref{answer-dual}. 

However, if $j$ is even, then $k_i=1$, i.e., $\alpha_{2i-1}^2=v_n\alpha_{2j-1}$, and 
\begin{multline*}
\Big\langle\Delta(e_{2j-1}),\,\alpha_{2i-1}\gamma_{2^{k_h-1}}(\alpha_{2h-1})\otimes\gamma_{2^{k_h-1}}(\alpha_{2h-1})\Big\rangle=v_n^2=\\
=\Big\langle\Delta(e_{2j-1}),\,\gamma_{2^{k_h-1}}(\alpha_{2h-1})\otimes\gamma_{2^{k_h-1}}(\alpha_{2h-1})\alpha_{2i-1}\Big\rangle
\end{multline*}
for $2i-1=(2h-1)2^{k_h}-2^{n}+1$. Let us denote $2j-1=2^n-1-2k$. Then we have $2i-1=2^{n}-1-k$, and $2h-1=2^{n}-1-\frac k2$. Moreover, if $k$ is divisible by $4$, we see that $h$ is even and we can continue this process by induction. 

Denoting $\langle t\rangle=2^n-1-t$, we obtain formula~(\ref{hereitis}).
\end{proof}

We can now compute $\mathrm{CK}(n)^*(\mathrm{SO}_m)$ for $m\leq2^{n+1}$ as a corollary of the above proposition.

\begin{prop}
\label{co-mult-ck}
For $m\leq2^{n+1}$ the algebra structure of $\mathrm{CK}(n)^*(\mathrm{SO}_m)$ is given by
$$
\mathbb F_2[v_n][e_1,e_2,\ldots,e_s]/(e_i^2-e_{2i})
$$
where $s=\lfloor\frac{m-1}{2}\rfloor$ and $e_{k}$ stands for $0$ if $k>s$. The reduced co-multiplication $\widetilde\Delta(x)=\Delta(x)-x\otimes1-1\otimes x$ is given by
$$
\widetilde\Delta(e_{\langle 2k\rangle})=\sum_{i=0}^{\nu_2(k)}v_n^{i+1}\,e_{\langle k/2^{i}\rangle}\otimes e_{\langle k/2^{i}\rangle}\,\prod_{j=0}^{i-1}\left(e_{\langle k/2^{j}\rangle}\otimes1+1\otimes e_{\langle k/2^{j}\rangle}\right)
$$
where $\langle t\rangle$ stands for $2^n-1-t$, $k>0$, and $\nu_2(k)$ is the $2$-adic valuation of $k$, i.e., $k/2^{\,\nu_2(k)}$ is an odd integer, and
$$
\widetilde\Delta(e_{2^n-1})=v_n\,e_{2^{n}-1}\otimes e_{2^{n}-1}.
$$
Moreover, $e_1$ coincides with the first Chern class $c_1^{\mathrm{CK}(n)}(e_1^{\mathrm{CH}})$ of the generator $e_1^{\mathrm{CH}}$ of $\mathrm{CH}^1(\mathrm{SO}_m;\mathbb Z)$.
\end{prop}

\begin{proof}
First, assume that $m$ is odd. We know that the natural map
$$
\mathrm{Ch}^*(\mathrm{SO}_{2^{n+1}-1})\rightarrow\mathrm{Ch}^*(\mathrm{SO}_{m})
$$
coincides with taking a quotient modulo $e_k$ for $k>\lfloor\frac{m-1}{2}\rfloor$. Therefore the corresponding map on the connective Morava K-theories
$$
\mathrm{CK}(n)^*(\mathrm{SO}_{2^{n+1}-1})\rightarrow\mathrm{CK}(n)^*(\mathrm{SO}_{m})
$$
is surjective by Nakayama's Lemma (see~\cite[Lemma~5.5]{LPSS}), and the images of $e_k$ become divisible by $v_n$ for $k>\lfloor\frac{m-1}{2}\rfloor$. 

Denote the image of $e_i$ in $\mathrm{CK}(n)^*(\mathrm{SO}_{m})$ by $\overline e_i$ and let
$$
\overline e_k=v_n^aP+v_n^{a+1}R
$$
for some $P$ not divisible by $v_n$. We will show that such an equation is impossible by formal reasons, therefore in fact $\overline e_k=0$ for $k>\lfloor\frac{m-1}{2}\rfloor$.

Let $k=2^t(2^n-1-2j)$ and let $\langle t\rangle$ denote $2^n-1-t$. Then 
$$
\Delta(e_k)=\Delta(e_{\langle 2j\rangle}^{2^t})=e_{\langle 2j\rangle}^{2^t}\otimes1+1\otimes e_{\langle 2j\rangle}^{2^t}+\sum_i v_n^{2^t(i+1)}\,e_{\langle j/2^i\rangle}^{2^t}\otimes e_{\langle j/2^i\rangle}^{2^t}\,S_i,
$$ 
and therefore
\begin{align}
\label{strashno}
\overline e_k\otimes1+1\otimes\overline e_k+\sum_iv_n^{2^t(i+1)}\,\overline e_{\langle j/2^i\rangle}^{\,2^t}\otimes\overline e_{\langle j/2^i\rangle}^{\,2^t}\,\overline S_i=v_n^a\Delta(P)+v_n^{a+1}\Delta(R).
\end{align}
For $k=2^n-1$ we obtain 
$$
\overline e_{2^n-1}\otimes1+1\otimes\overline e_{2^n-1}+v_n^{}\,\overline e_{2^n-1}\otimes\overline e_{2^n-1}=v_n^a(P\otimes1+1\otimes P)+v_n^{a+1}S,
$$
which implies that the image of $P$ in $\mathrm{Ch}^*(\mathrm{SO}_{m})$ is primitive and therefore coincides with $\sum e_{2h-1}^{2^{d_h}}$ for some $0\leq d_h<\lfloor\log_2\left(\frac{m-1}{2h-1}\right)\rfloor$ (cf.~\cite[Lemma~6.1\,a)]{LPSS}). Then 
$$
\overline e_{2^n-1}=v_n^a\sum\overline e_{2h-1}^{\,2^{d_h}}+v_n^{a+1}R',
$$
and comparing the gradings we have
$$
2^n-1=a(1-2^n)+2^{d_h}(2h-1)
$$
for each $h$. However, $a\geq1$ and $2^{d_h}(2h-1)\leq \frac{m-1}{2}$, therefore the above identity is impossible. This shows that $\overline e_{2^n-1}=0$. 

Take the largest $k$ such that 
$$
\overline e_k=\overline e_{[2j]}^{2^t}\neq0.
$$ 
Then by~(\ref{strashno}) we conclude that
$$
\overline e_k\otimes1+1\otimes\overline e_k=v_n^a\Delta(P)+v_n^{a+1}\Delta(R),
$$
and the image of $P$ in $\mathrm{Ch}^*(\mathrm{SO}_{m})$ is primitive. Then
$$
\overline e_{k}=v_n^a\sum\overline e_{2h-1}^{\,2^{d_h}}+v_n^{a+1}R'
$$
as above, and
$$
k=a(1-2^n)+2^{d_h}(2h-1)
$$
for each $h$. But again $a\geq1$, $k>\lfloor\frac{m-1}{2}\rfloor$ and $2^{d_h}(2h-1)\leq \frac{m-1}{2}$, therefore the above identity is impossible. 

This shows that all $\overline e_k=0$ for $k>\lfloor\frac{m-1}{2}\rfloor$. Comparing the ranks we obtain the claim for odd $m$ (cf.~\cite[Theorem~6.10]{LPSS}).

Now assume that $m$ is even. As we already remarked, we can assume that $k=\mathbb C$, in particular, we can consider the natural map $\mathrm{SO}_{m-1}\rightarrow\mathrm{SO}_m$ induces a surjection on Chow rings
$$
\mathrm{CH}^*(\mathrm{SO}_m)\rightarrow\mathrm{CH}^*(\mathrm{SO}_{m-1})
$$
(in fact, an isomorphism), cf., e.g.,~\cite[Section~2]{Pit}, and therefore a surjection on $\mathrm{CK}(n)^*$ by Nakayama's Lemma~\cite[Lemma~5.5]{LPSS}. Comparing the ranks we conclude that
$$
\mathrm{CK}(n)^*(\mathrm{SO}_m)\rightarrow\mathrm{CK}(n)^*(\mathrm{SO}_{m-1})
$$
is an isomorphism of Hopf algebras (cf.~\cite[Theorem~6.10]{LPSS}).

Finally, we will show that $e_1$ coincides with the first Chern class of a certain bundle. Let $e_1^{\mathrm{CH}}$ denote a generator of $\mathrm{CH}^1(\mathrm{SO}_m;\,\mathbb Z)$, and $e_1'=c_1^{\mathrm{CK}(n)}(e_1^{\mathrm{CH}})$. The images of $e_1$ and $e_1'$ coincide in $\mathrm{Ch}^*(\mathrm{SO}_m)$, and
$$
\Delta(e_1')=\mathrm{FGL}_{\mathrm{CK}(n)}(e_1'\otimes1,\,1\otimes e_1'),
$$
see the proof of~\cite[Lemma~6.12]{LPSS}. Since $e_{2i-1}^{2^n}=0$ in $\mathrm{CK}(n)^*(\mathrm{SO}_m)$ for all $i$, we conclude that $(e_1')^{2^n}=0$, and therefore
$$
\widetilde\Delta(e_1')=v_n\,(e_1')^{2^{n-1}}\otimes(e_1')^{2^{n-1}}
$$
by, e.g.,\cite[Subsection~2.3]{PSlog}. Let 
$$
e_1'-e_1=v_n^aP+v_n^{a+1}R,
$$
for some $P$ not divisible by $v_n$. Again, we will show that the above equation is impossible by formal reasons.

Indeed, one has
\begin{multline*}
\Delta(e_1')=(e_1+v_n^aP+v_n^{a+1}R)\otimes1+1\otimes(e_1+v_n^aP+v_n^{a+1}R)+\\
+v_n(e_1+v_n^aP+v_n^{a+1}R)^{2^{n-1}}\otimes(e_1+v_n^aP+v_n^{a+1}R)^{2^{n-1}}=\\
=\Delta(e_1)+v_n^a(P\otimes1+1\otimes P)+v_n^{a+1}S.
\end{multline*}
On the other hand,
$$
\Delta(e_1')=\Delta(e_1+v_n^aP+v_n^{a+1}R)=\Delta(e_1)+v_n^a\,\Delta(P)+v_n^{a+1}\Delta(R).
$$
This implies that $P$ becomes primitive modulo $v_n$, i.e., 
$$
e_1'-e_1=v_n^a\sum e_{2h-1}^{\,2^{d_h}}+v_n^{a+1}R',
$$
and
$$
1=a(1-2^n)+2^{d_h}(2h-1),
$$
where $a\geq1$ and $2^{d_h}(2h-1)\leq \frac{m-1}{2}\leq2^n-1$. Therefore the above identity is impossible and $e_1'=e_1$.
\end{proof}

Finally, we will compute the co-multiplication in $\mathrm{CK}(n)^*(\mathrm{SO}_m)$ for $m>2^{n+1}$.
\begin{tm}
\label{co-mult-ck-big}
For all $m$, $n\in\mathbb N\setminus0$ the algebra structure of $\mathrm{CK}(n)^*(\mathrm{SO}_m)$ is given by
$$
\mathbb F_2[v_n][e_1,e_2,\ldots,e_s]/(e_i^2=e_{2i}\ \forall i, v_ne_i=0\text{ for }i\geq2^n),
$$
where $s=\lfloor\frac{m-1}{2}\rfloor$ and $e_{k}$ stands for $0$ if $k>s$. The reduced co-multiplication $\widetilde\Delta(x)=\Delta(x)-x\otimes1-1\otimes x$ is given by
$$
\widetilde\Delta(e_{\langle 2k\rangle})=\sum_{i=0}^{\nu_2(k)}v_n^{i+1}\,e_{\langle k/2^{i}\rangle}\otimes e_{\langle k/2^{i}\rangle}\,\prod_{j=0}^{i-1}\left(e_{\langle k/2^{j}\rangle}\otimes1+1\otimes e_{\langle k/2^{j}\rangle}\right)
$$
where $\langle t\rangle$ stands for $2^n-1-t$, $0<k<2^{n-1}$, and $\nu_2(k)$ is the $2$-adic valuation of $k$, i.e., $k/2^{\,\nu_2(k)}$ is an odd integer, 
$$
\widetilde\Delta(e_{2^n-1})=v_n\,e_{2^{n}-1}\otimes e_{2^{n}-1},
$$
and
$$
\widetilde\Delta(e_{2k-1})=0
$$
for $k>2^{n-1}$. 
Moreover, $e_1$ coincides with the first Chern class $c_1^{\mathrm{CK}(n)}(e_1^{\mathrm{CH}})$ of the generator $e_1^{\mathrm{CH}}$ of $\mathrm{CH}^1(\mathrm{SO}_m;\mathbb Z)$.
\end{tm}
\begin{proof}
For $m\leq2^{n+1}$ the claim follows from Proposition~\ref{co-mult-ck}. Assume that $m>2^{n+1}$ and denote $m_0=2^{n+1}-1$ for $m$ odd, and $m_0=2^{n+1}$ for $m$ even. 

The claim about the algebra structure is proven in Theorem~\ref{ck-algebra}. Moreover, from the proof of Theorem~\ref{ck-algebra} it follows that one can take any lift of $e_1\in\mathrm{CK}(n)^*(\mathrm{SO}_{m_0})$ as a generator $e_1\in\mathrm{CK}(n)^*(\mathrm{SO}_{m})$. Therefore we can assume that $e_1=c_1^{\mathrm{CK}(n)}(e_1^{\mathrm{CH}})$ is the first Chern class of the generator $e_1^{\mathrm{CH}}$ of $\mathrm{CH}^1(\mathrm{SO}_m;\mathbb Z)$.

Denote $H=\mathrm{CK}(n)^*(\mathrm{SO}_{m})$, $I=(e_{2^n},e_{2^n+1},\ldots,e_s)\trianglelefteq H$, and 
$$
H_0=\mathrm{CK}(n)^*(\mathrm{SO}_{m_0})=H/I.
$$
Then one also has
$$
H\otimes_{\mathbb F_2[v_n]}H/(H\otimes_{\mathbb F_2[v_n]}I+I\otimes_{\mathbb F_2[v_n]}H)=H_0\otimes_{\mathbb F_2[v_n]}H_0.
$$
Observe that all generators of $I$ have degree at least $2^n$ and are annihilated by $v_n$. This implies that $H\otimes_{\mathbb F_2[v_n]}I+I\otimes_{\mathbb F_2[v_n]}H$ does not have non-zero elements of degree less than $2^n$.

For $0<k<2^{n-1}$ we conclude by Proposition~\ref{prop-answer} that
$$
\widetilde\Delta(e_{\langle 2k\rangle})-\sum_{i=0}^{\nu_2(k)}v_n^{i+1}\,e_{\langle k/2^{i}\rangle}\otimes e_{\langle k/2^{i}\rangle}\,\prod_{j=0}^{i-1}\left(e_{\langle k/2^{j}\rangle}\otimes1+1\otimes e_{\langle k/2^{j}\rangle}\right)
$$
is a (homogeneous) element of $H\otimes_{\mathbb F_2[v_n]}I+I\otimes_{\mathbb F_2[v_n]}H$ of degree $\langle 2k\rangle<2^n$. Therefore this element is $0$. Similarly, $\widetilde\Delta(e_{2^n-1})-v_n\,e_{2^{n}-1}\otimes e_{2^{n}-1}=0$.

It remains to compute $\Delta(e_{2k-1})$ for $k>2^{n-1}$. Let $I'=(e_{2k-1},e_{2k},\ldots,e_s)$ and $H':=H/I'=\mathrm{CK}(n)^*(\mathrm{SO}_{m'})$ for $m'=4k-3$ for $m$ odd or $m'=4k-2$ for $m$ even. Observe that $\Delta(e_{2k-1})$ lies in the kernel of the map
$$
H\otimes_{\mathbb F_2[v_n]}H\rightarrow H'\otimes_{\mathbb F_2[v_n]}H',
$$
and the latter kernel coincides with $H\otimes_{\mathbb F_2[v_n]}I'+I'\otimes_{\mathbb F_2[v_n]}H$. However, $\Delta(e_{2k-1})$ is homogeneous of degree $2k-1$, and the only non-zero element of $I'$ of degree less than or equal to $2k-1$ is $e_{2k-1}$ itself. This implies that $\Delta(e_{2k-1})$ is a linear combination of $1\otimes e_{2k-1}$ and $e_{2k-1}\otimes1$ which finishes the proof.
\end{proof}

Now Theorem~\ref{co-mult} also follows.

\section{Morava $J$-invariant for orthogonal groups}
\label{section-j}

\subsection{A generalized $J$-invariant}

Let $G=\mathrm{SO}_{m}$ denote a (split) special orthogonal group and let $E\in\mathrm H^1(k,\,G)$ be a torsor corresponding to a quadratic form $q\colon V\rightarrow k$ (with trivial discriminant), see~\cite[(29.29)]{KMRT}. For a free theory $A^*$ the extension of scalars $\overline k/k$ induces a natural map $A^*(E)\rightarrow A^*(G)$. Following~\cite[Definition~4.6]{PShopf} we give the following definition.
\begin{df}
\label{j-inv-def}
In the notation above, we call
$$
H^*_A(E)=H^*_A(q)=A^*(G)\otimes_{A^*(E)}A^*(\mathrm{pt})
$$
the (generalized) {\it $J$-invariant} of $E$ (or $q$) corresponding to the theory $A^*$.
\end{df}
By~\cite[Lemma~4.5]{PShopf}, $H^*_A(E)$ is a quotient bi-algebra of $A^*(G)$. Therefore after classification of bi-ideals in $\mathrm{CK}(n)^*(G)$ obtained in Subsection~\ref{sat-bi-id} we will be able to prove the following result.

\begin{tm}
\label{j-ch-ck}
For all $n$, $m\in\mathbb N\setminus0$ and $E\in\mathrm H^1(k,\,\mathrm{SO}_m)$ one has
$$
H^*_{\mathrm{Ch}}(E)=H^*_{\mathrm{CK}(n)}(E)/v_n,\quad H^*_{\mathrm{CK}(n)}(E)=H^*_{\mathrm{Ch}}(E)\otimes_{\mathbb F_2}\mathbb F_2[v_n]/(v_n\overline e_i=0\text{ for }i\geq2^n)
$$
where $\overline e_i$ denotes the image of $e_i\in\mathrm{CK}(n)^*(\mathrm{SO}_m)$ in $H^*_{\mathrm{CK}(n)}(E)$ {\rm(}see Theorem~\ref{co-mult-ck-big}{\rm)}. 
\end{tm} 
In other words, the $J$-invariants for the Chow theory $\mathrm{Ch}^*$ and  for the connective Morava K-theory $\mathrm{CK}(n)^*$ carry the same information. As a consequence we clearly obtain the following corollary.

\begin{cl}
\label{j-ch-k}
For $n$, $m\in\mathbb N\setminus0$ and $E\in\mathrm H^1(k,\,\mathrm{SO}_m)$ one has
$$
H^*_{\mathrm{K}(n)}(E)=H^*_{\mathrm{Ch}}(E)\otimes_{\mathbb F_2}\mathbb F_2[v_n^{\pm1}]/(\overline e_i=0\text{ for }i\geq2^n)
$$
where $\overline e_i$ denotes the image of $e_i\in\mathrm{K}(n)^*(\mathrm{SO}_m)$ in $H^*_{\mathrm{K}(n)}(E)$ {\rm(}see Theorem~\ref{co-mult}{\rm)}.
\end{cl}
In other words, the $J$-invariant for the periodic Morava K-theory can be obtained from the usual $J$-invariant for $\mathrm{Ch}^*$ by an appropriate truncation.

\subsection{Saturated bi-ideals in $\mathrm{CK}(n)^*(\mathrm{SO}_m)$}
\label{sat-bi-id}

We call a submodule $N$ in a $\mathbb F_2[v_n]$-module $M$ {\it saturated}, if $v_nx\in N\Rightarrow x\in N$ for any $x\in M$. We will need the following technical result.

\begin{lm}
\label{sublemma}
1{\rm)} Let $M$ be a graded, finitely generated $\mathbb F_2[v_n]$-module, and $N\leq M$ be its {\rm(}graded{\rm)} submodule. Then $N$ is saturated if and only if $M/N$ is free.\\
2{\rm)} Let $H$ be a graded, finitely generated $\mathbb F_2[v_n]$-algebra, and $I$ its {\rm(}homogeneous{\rm)}  saturated ideal. Then $I\otimes H+H\otimes I$ is a saturated ideal of $H\otimes H$.
\end{lm}
\begin{proof}
The first claim is obvious. To get 2), one only has to observe that
$$
H\otimes H/(I\otimes H+H\otimes I)\cong (H/I)\otimes(H/I)
$$
is free, and use 1).
\end{proof}

We will also need the following well-known corollary of the Milnor--Moore Theorem~\cite[Theorem~7.16]{MM}. We provide a proof in Section~\ref{app-milnor-moore} for the convenience of the reader.

\begin{lm}
\label{hopf-chow}
Let $H=\mathbb F_{2}[e_1,\,e_3,\ldots,e_{2r-1}]/(e_{1}^{2^{k_1}},\,e_{3}^{2^{k_2}},\ldots,\,e_{2r-1}^{2^{k_r}})$ be a commutative, graded Hopf algebra with $e_{2i-1}$ primitive of degree $2i-1$. Then any {\rm(}homogeneous{\rm)}  bi-ideal $I$ of $H$ coincides with the ideal
$$
\big(e_1^{2^{a_1}},\,e_3^{2^{a_2}},\ldots,e_{2r-1}^{2^{a_{r}}}\big)
$$
for some $0\leq a_{i}\leq k_i$.
\end{lm}

For $m\leq 2^{n+1}$ we will give a description of {\it saturated}\, bi-ideals in $\mathrm{CK}(n)^*(\mathrm{SO}_m)$ analogues to Lemma~\ref{hopf-chow}. First, we will need the following technical result (cf.~\cite[Lemma~6.11]{LPSS}).

\begin{lm}
\label{impossible-equation}
Let  $n$, $m$, $j$, $k$ be positive integers such that $m\leq2^{n+1}$, and $j,\,k\leq \left\lfloor\frac{m+1}{4}\right\rfloor$. Then the equation
$$
(1-2^n)x+(2j-1)2^y=(2k-1)2^z
$$
has no integral solutions $x>0$, $0\leq y<\left\lfloor\mathrm{log}_2\left(\frac{m-1}{2j-1}\right)\right\rfloor$, $z\geq0$. 
\end{lm}
\begin{proof}
Since 
$
(2j-1)2^y=(2k-1)2^z+(2^n-1)x\geq 2^n,
$ 
we have 
$$
2^{y+1}\geq\frac{2^{n+1}}{2j-1}\geq\frac{m}{2j-1}>\frac{m-1}{2j-1}.
$$
In other words $y+1>\left\lfloor\mathrm{log}_2\left(\frac{m-1}{2j-1}\right)\right\rfloor$ contradicting the assumptions.
\end{proof}

\begin{prop}
\label{hopf-ck}
Let $$H=\mathrm{CK}(n)(\mathrm{SO}_m)\cong\mathbb F_{2}[v_n][e_1,\,e_3,\ldots,e_{2r-1}]/(e_{1}^{2^{k_1}},e_{3}^{2^{k_2}},\ldots,e_{2r-1}^{2^{k_r}})$$ for $m\leq2^{n+1}$, $r=\lfloor\frac{m+1}{4}\rfloor$, $k_i=\left\lfloor\mathrm{log}_2\left(\frac{m-1}{2i-1}\right)\right\rfloor$ {\rm(}see Proposition~\ref{co-mult-ck}{\rm)}. Then any homogeneous {\it saturated} bi-ideal $I$ of $H$ coincides with the ideal
$$
\big(e_1^{2^{a_1}},\,e_3^{2^{a_2}},\ldots,e_{2r-1}^{2^{a_r}}\big)
$$
for some $0\leq a_i\leq k_i$.
\end{prop}

\begin{proof}
Let $\overline{(\ )}$ denote the image modulo $v_n$. Since $\overline I$ is also a bi-ideal, we conclude by Lemma~\ref{hopf-chow} that 
$$
\overline I=\big(\overline e_1^{2^{a_1}},\,\overline e_3^{2^{a_2}},\ldots,\overline e_{2r-1}^{2^{a_r}}\big)
$$
for some $0\leq a_i\leq k_i$. We claim that 
$$
I=\big(e_1^{2^{a_1}},\,e_3^{2^{a_2}},\ldots,e_{2r-1}^{2^{a_r}}\big)
$$
with the same indices $a_k$. In fact, it is enough to show that $e_k^{2^{a_k}}\in I$ for all $k$, and comparing dimensions we will get the claim.

Assume that $e_k^{2^{a_k}}\not\in I$ for some $k$, fix an arbitrary lift $E_{2k-1}$ of $\overline e_{2k-1}^{2^{a_{k}}}$ to an element of $I$, and write
\begin{align}
\label{E-formula}
E_{2k-1}=e_{2k-1}^{2^{a_k}}+v_n^aP
\end{align}
for $\overline P\neq0\in\overline H$. We will show that it is possible to find another lift $E_{2k-1}'\in I$ of $\overline e_{2k-1}^{2^{a_k}}$ with greater $a$ in the decomposition~(\ref{E-formula}). Continuing this process we will eventually show that $P=0$ by the degree reasons.

{\bf Case~1.} First, we will show that if $\overline e_{2k-1}^{\,2^e}\in\overline I$ for some $e>0$, then in fact $e_{2k-1}^{\,2^e}\in I$. 

Let $E_{2k-1}$ denote a pre-image of $\overline e_{2k-1}^{2^{e}}$ in $I$, and write
\begin{align}
\label{E-formula-even}
E_{2k-1}=e_{2k-1}^{2^{e}}+v_n^aP
\end{align}
for $\overline P\neq0\in\overline H$. Observe that $e_{2k-1}^{2^{e}}$ is primitive in $H$ (see Example~\ref{primitive}), then, applying $\widetilde\Delta$ to~(\ref{E-formula-even}) we get
$$
v_n^a\widetilde\Delta(P)=\widetilde\Delta(E_{2k-1})
\in I\otimes H+H\otimes I
$$
since $I$ is a bi-ideal. By Lemma~\ref{sublemma} we conclude that
$$
\widetilde\Delta(P)
\in I\otimes H+H\otimes I,
$$
i.e., the class of $\overline P$ is primitive in $\overline H/\,\overline I$, and therefore $\overline P$ is congruent modulo $\overline I$ to $\sum_{i\in\mathcal{I}} \overline e_{2i-1}^{2^{s_i}}$ for some $0\leq s_i<a_i$ and some index set 
$\mathcal{I}\subseteq\{1,\ldots,r\}$ (see~\cite[Lemma~6.1\,a)]{LPSS}). Comparing the codimensions we obtain an equation
$$
(2k-1)2^{e}=(1-2^n)a+(2i-1)2^{s_i}.
$$
However, this equation has no solutions by Lemma~\ref{impossible-equation}. Therefore, $\mathcal{I}=\emptyset$ and, thus, $\overline P\in\overline I$. 

Take a pre-image $Q\in I$ of $\overline P$, then $P-Q$ is divisible by $v_n$, i.e., there exist $P'\in H$ such that $P-Q=v_nP'$. Therefore
$$
E_{2k-1}':=E_{2k-1}-v_n^aQ=e_{2k-1}^{2^{e}}+v_n^aP-v_n^aQ=e_{2k-1}^{2^{e}}+v_n^{a+1}P'\in I,
$$
i.e., we obtain a pre-image of $\overline e_{2k-1}^{\,2^{e}}$ in $I$ with a greater $a$ than in~(\ref{E-formula-even}). Continuing this process, we will eventually show that $e_{2k-1}^{2^{e}}\in I$.

{\bf Case~2.} By Case~1 we can assume that $a_k=0$. We use the notation $\langle t\rangle=2^n-1-t$ from Proposition~\ref{co-mult-ck}, and denote $\langle 2j\rangle=2k-1$. We can assume, moreover, that $k$ is the largest possible (i.e., $j$ is the least possible) such that $a_k=0$ and $e_{2k-1}\not\in I$. 

Our aim is to prove that $\overline P\in\overline I\otimes\overline H+\overline H\otimes\overline I$ and use the same argument as in Case~1 to find a presentation as~\eqref{E-formula} with a greater $a$.

Applying $\widetilde\Delta$ to~(\ref{E-formula}), we obtain
$$
v_n\,e_{\langle j\rangle}\otimes e_{\langle j\rangle}+v_n^2R+v_n^a\widetilde\Delta(P)
\in I\otimes H+H\otimes I,
$$
where $v_n^2R=\widetilde\Delta(e_{\langle 2j\rangle})-v_ne_{\langle j\rangle}\otimes e_{\langle j\rangle}$ is determined by Theorem~\ref{co-mult-ck-big}. 

{\bf Subcase~2.1.} For $j=0$ (i.e., $2k-1=2^n-1$) observe that $R=0$. Using that $E_{2^n-1}\otimes E_{2^n-1}\in I\otimes I$ we obtain by~\eqref{E-formula} that
$$ 
v_n^a\widetilde\Delta(P)+v_n^{a+1}R'\in I\otimes H+H\otimes I
$$
for $R'=e_{2^n-1}\otimes P+P\otimes e_{2^n-1}+v_n^{a}P\otimes P$. As above, this implies that $\overline P$ becomes primitive modulo $\overline I$, however, $\overline H/\,\overline I$ does not have primitive elements of the required degree by Lemma~\ref{impossible-equation}. This implies that $\overline P\in\overline I$, and therefore we can find a new lift $E_{2^n-1}'$ for $\overline e_{2^n-1}$ with greater $a$ in the decomposition~(\ref{E-formula}). Continuing this process we will eventually prove that $e_{2^n-1}\in I$.

{\bf Subcase~2.2.} Now consider the case $j>0$. The proof in this case is divided in three steps.

{\it Step~1.} First, we show that that $\overline e_{\langle j\rangle}\in\overline I$. 
Indeed, if $a>1$, we obtain that $\overline e_{\langle j\rangle}\otimes\overline e_{\langle j\rangle}\in\overline I\otimes\overline H+\overline H\otimes\overline I$ by Lemma~\ref{sublemma},
and therefore $\overline e_{\langle j\rangle}\in\overline I$. 
If $a=1$, we similarly obtain that $\overline e_{\langle j\rangle}\otimes\overline e_{\langle j\rangle}+\widetilde\Delta(\overline P)\in\overline I\otimes\overline H+\overline H\otimes\overline I$. 

Consider the natural basis of $\overline H$ consisting of the monomials in $\overline e_{2i-1}$. Then the part of this basis divisible by $\overline e_{2i-1}^{\,2^{a_i}}$ for some $i$ is a basis of $\overline I$, and the tensor product $\overline H\otimes\overline H$ has a natural basis consisting of decomposable tensors of basis elements. 

Observe that $\widetilde\Delta(\overline P)$ cannot have a monomial $\overline e_{\langle j\rangle}\otimes\overline e_{\langle j\rangle}$ in its decomposition as a sum of basis elements (this follows from the facts that $\overline e_{2i-1}$ are primitive, and ${2n\choose n}\equiv0\mod2$). Therefore we obtain that $\overline e_{\langle j\rangle}\otimes\overline e_{\langle j\rangle}\in\overline I\otimes\overline H+\overline H\otimes\overline I$, and $\overline e_{\langle j\rangle}\in\overline I$.

{\it Step~2.} Now we will show that in fact $e_{\langle j\rangle}\in I$. 
Recall that we chose $k$ in such a way that $2k-1=\langle2j\rangle$ is the largest possible number such that $a_k=0$ and $e_{2k-1}\not\in I$. Since $\overline e_{\langle j\rangle}\in\overline I$ we conclude that $\overline e_{\langle j\rangle}$ is equal to $\overline e_{2i-1}^{\,2^{a_i+e}}$ for some $i$ and $e\geq0$. However, if $a_i+e>0$ we already proved that $e_{2i-1}^{\,2^{a_i+e}}\in I$ (see Case~1). But if $a_i=0=e$ we know that $e_{2i-1}\in I$ by the choice of $k$. In both cases, we obtain that $e_{\langle j\rangle}\in I$.

{\it Step~3.} Finally, we show that $e_{\langle 2j\rangle}\in I$. Since $e_{\langle j\rangle}\in I$ by Step~2, we conclude that 
$$
\widetilde\Delta(e_{\langle2j\rangle})\in e_{\langle j\rangle}\otimes H+H\otimes e_{\langle j\rangle}\subseteq I\otimes H+H\otimes I
$$
by Theorem~\ref{co-mult-ck-big}. Consequently, applying $\widetilde\Delta$ to~(\ref{E-formula}), we have
$$ 
v_n^a\widetilde\Delta(P)\in I\otimes H+H\otimes I,
$$
and therefore $\widetilde\Delta(P)\in I\otimes H+H\otimes I$ by Lemma~\ref{sublemma}. As above, this implies that $\overline P$ becomes primitive modulo $\overline I$, however, $\overline H/\,\overline I$ does not have primitive elements of the required degree by Lemma~\ref{impossible-equation}. This implies that $\overline P\in\overline I$, and therefore we can find a new lift $E_{2k-1}'$ for $\overline e_{2k-1}$ with greater $a$ in the decomposition~(\ref{E-formula}). Continuing this process we will eventually prove that $e_{2k-1}=e_{\langle 2j\rangle}\in I$. This finishes the proof of Subcase~2.2.

We proved that $e_{2i-1}^{2^{a_i}}\in I$, and it remains to observe that they clearly generate $I$ since it is true modulo $v_n$ and $I$ is saturated.
\end{proof}

\subsection{The $J$-invariant for the connective Morava K-theory}

Let $n$, $m\in\mathbb N\setminus0$. For $m\leq 2^{n+1}$ we put $m_0=m$, and for $m>2^{n+1}$ let $m_0=2^{n+1}-1$ for an odd $m$ and $m_0=2^{n+1}$ for an even $m$. Denote 
$$
H=\mathrm{CK}(n)^*(\mathrm{SO}_m)\quad\text{ and }\quad H_0=\mathrm{CK}(n)^*(\mathrm{SO}_{m_0}).
$$ 

For a smooth irreducible variety $X$ and a free theory $A^*$ the $A^*(\mathrm{pt})$-algebra $A^*(X)$ is naturally augmented via
$$
A^*(X)\xrightarrow{\eta^A}A^*(\mathrm{Spec}\,k(X))\xrightarrow{\cong}A^*(\mathrm{pt})
$$
where $\eta$ denotes the generic point of $X$ (see~\cite[Remark~1.2.12]{LM}). We denote by $A^*(X)^+$ the augmentation ideal.

For $E\in\mathrm H^1(k,\,\mathrm{SO}_m)$ consider the natural map $\mathrm{CK}(n)^*(E)\rightarrow\mathrm{CK}(n)^*(\mathrm{SO}_m)$ induced by the restriction of scalars. Denote by $I$ the ideal of $H$ generated by the image of the map
$$
\mathrm{CK}(n)^*(E)^+\rightarrow H,
$$
and denote by $I_0$ the ideal of $H_0$ generated by the image of 
$$
\mathrm{CK}(n)^*(E)^+\rightarrow H\rightarrow H_0
$$
where the map $H\rightarrow H_0$ is the pullback along the natural closed embedding $\mathrm{SO}_{m_0}\hookrightarrow\mathrm{SO}_{m}$. In other words, $H/I$ is the generalized $J$-invariant of $E$ in the sense of Definition~\ref{j-inv-def}, and $I$ is a bi-ideal of $H$ by~\cite[Lemma~4.5]{PShopf}. 
\begin{lm}
\label{j-is-saturated}
In the notation above, $I_0$ is a saturated bi-ideal of $H_0$.
\end{lm}
\begin{proof}
It is easy to see that $I_0$ is also a bi-ideal of $H_0$ as an image of $I$.

Let $J_0$ denote the saturation of $I_0$, i.e., $$J_0=\{x\in H\mid\exists\,i\geq0,\  v_n^ix\in I_0\}.$$ For $x\in J_0$ one has
$$
v_n^i\Delta(x)=\Delta(v_n^ix)\in I_0\otimes H+H\otimes I_0\subseteq J_0\otimes H+H\otimes J_0,
$$ 
and therefore $J_0$ is a bi-ideal by Lemma~\ref{sublemma}. Then, by Proposition~\ref{hopf-ck}, $$J_0=\big(e_1^{2^{a_1}},\,e_3^{2^{a_2}},\ldots,e_{2r-1}^{2^{a_r}}\big).
$$
We will show that in fact $e_{2i-1}^{2^{a_i}}\in I_0$.

Since $v_n^te_{2i-1}^{2^{a_i}}\in I_0$ there exist homogeneous elements $h_k\in H_0$ and
$$
x_k\in\mathrm{Im}\big(\mathrm{CK}(n)(E)^+\rightarrow H\rightarrow H_0\big)
$$
such that $v_n^te_{2i-1}^{2^{a_i}}=\sum x_kh_k$. However, since the codimension of $e_{2i-1}^{2^{a_i}}$ is less than or equal to $2^n-1$, we can assume that for each $k$ either $x_k$ or $h_k$ has degree less than $0$ (if the degree of $x_kh_k$ is $0$ use additionally that $e_{2i-1}\in H^+_0$). By a theorem of Levine--Morel the algebraic cobordism is generated by the elements of non-negative degrees~\cite[Theorem~1.2.14]{LM}. This implies that either $x_k$ or $h_k$ is divisible by $v_n$. In other words, there exist $h_k'\in H_0$ and $x_k'\in\mathrm{Im}\big(\mathrm{CK}(n)(E)^+\rightarrow H\rightarrow H_0\big)$ such that 
$$
v_n^te_{2i-1}^{2^{a_i}}=\sum v_nx_k'h_k'.
$$
Since $H_0$ does not have $v_n$-torsion, we conclude that
$
v_n^{t-1}e_{2i-1}^{2^{a_i}}\in I_0.
$ 
Continuing this process we obtain the claim.
\end{proof}

Next we prove the following theorem.

\begin{tm}
\label{j-classification}
For $H$ and $H_0$ as above let $I$ be a {\rm(}homogeneous{\rm)} bi-ideal of $H$ such that its image $I_0$ in $H_0$ is saturated. Then 
$$
I=(e_1^{2^{a_1}},\ldots,e_{2r-1}^{2^{a_{r}}})
$$ 
for some $0\leq a_i\leq k_i$. In particular, $I$ is determined by its image modulo $v_n$.
\end{tm}
\begin{proof}
For $m\leq2^{n+1}$ the claim is proven in Proposition~\ref{hopf-ck}. We assume that $m>2^{n+1}$ and argue by induction.

We denote the reduction modulo $v_n$ by $\,\overline{\phantom{x}}$. Let 
$
\overline I=(\overline e_1^{\,2^{a_1}},\ldots,\overline e_{2r-1}^{\,2^{a_{r}}})
$ 
according to Lemma~\ref{hopf-chow}. Let $s=\lfloor\frac{m-1}{2}\rfloor$.\\
{\bf Case~1.} Assume that $e_s\neq e_{2i-1}^{2^{a_{i}}}$ for all $i$. Since $I/e_s$ is a bi-ideal in $H/e_s$ such that $(I/e_s)_0$ is saturated, using the induction we have
$$
I/e_s=(e_1^{2^{a_1}},\ldots,e_{2r-1}^{2^{a_r}}).
$$
Then there exist $E_1,\ldots,E_r\in I$ such that $E_i-e_{2i-1}^{2^{a_i}}=e_s\cdot x_i$. Since $e_s^2=0$ and $v_ne_s=0$ we can assume that $x_i$ are polynomials with $\mathbb F_2$-coefficients in $e_1,\ldots,e_{s-1}$. However 
$$
\overline e_s\cdot\overline x_i=\overline e_{2i-1}^{\,2^{a_i}}-\overline E_i\in\overline I,
$$
and this can only happen if $\overline x_i\in\overline I$. Then
$
x_i-\sum e_{2j-1}^{2^{a_j}}\cdot y_{ij} = v_nz_i
$ 
for some $y_{ij}$, $z_i\in H$. Then one has
\begin{multline*}
e_{2i-1}^{2^{a_i}}=E_i-e_s\left(v_nz+\sum e_{2j-1}^{2^{a_j}}\cdot y_{ij}\right)=\\=E_i-e_s\left(\sum (E_j-e_sx_j)\cdot y_{ij}\right)=E_i-e_s\left(\sum E_j\cdot y_{ij}\right)\in I.
\end{multline*}
It remains to prove that $I\subseteq(e_1^{2^{a_1}},\ldots,e_{2r-1}^{2^{a_r}})$. Take $x\in I$. Then $x-\sum e_{2i-1}^{2^{a_i}}\, x_{i}=e_sz$ for some $x_i$, $z\in H$. Again, we can assume that $z$ is a polynomial with $\mathbb F_2$-coefficients in $e_1,\ldots,e_{s-1}$. But $\overline e_s\,\overline z\in\overline I$, therefore $\overline z\in \overline I$, i.e., $z-\sum e_{2j-1}^{2^{a_j}}\, y_{j}=v_nw$ for some $y_j$, $w\in H$. Therefore
$$
x=\sum e_{2i-1}^{2^{a_i}}\, x_{i}+e_s\left(\sum e_{2j-1}^{2^{a_j}}\, y_{j}\right)
$$
which proves the claim.\\
{\bf Case~2.} Assume that $e_s= e_{2i_0-1}^{2^{a_{i_0}}}$ for some $i_0$. Again, $I/e_s$ is a bi-ideal in $H/e_s$ such that $(I/e_s)_0$ is saturated. Then using the induction we have
$$
I/e_s=(e_1^{2^{a_1}},\ldots,\widehat e_s,\ldots,e_{2r-1}^{2^{a_r}})
$$
where $\widehat e_s$ means that we excluded $e_s=e_{2i_0-1}^{2^{a_{i_0}}}$ from $e_1^{2^{a_1}},\ldots,e_{2r-1}^{2^{a_r}}$. As above, we can find $E_i\in I$ such that $E_i-e_{2i-1}^{2^{a_i}}=e_sx_i$ for $i\neq i_0$ and some $x_i\in H$, and $e_s+v_ny\in I$ for some $y\in H$. Since the image of $e_s+v_ny$ coincides with the image of $v_ny$ in $I_0$, and the latter ideal is saturated, we conclude that the image of $y$ lies in $I_0$. Then there exist $y_i\in H$ such that the image of $Y=\sum e_{2i-1}^{2^{a_i}}\,y_i$ in $H_0$ coincides with the image of $y$. Obviously, the image of $Y'=\sum E_i\,y_i$ also coincides with the image of $Y$. But the kernel of $H\rightarrow H_0$ is annihilated by $v_n$, in other words,
$$
v_n(y-Y')=0\in H.
$$
But then $e_s=(e_s+v_ny)-v_nY'\in I$. Therefore, $e_{2i-1}^{2^{a_i}}=E_i-e_s\,x_i\in I$. 

Finally, it remains to show that $I\subseteq(e_1^{2^{a_1}},\ldots,e_{2r-1}^{2^{a_r}})$. Let $x\in I$. Then 
$$
x-\sum e_{2i-1}^{2^{a_i}}\,z_i=e_s\,w
$$ for some $z_i$, $w\in H$ by induction. The claim follows.
\end{proof}

Now the proof of Theorem~\ref{j-ch-ck} follows from Lemma~\ref{j-is-saturated} and Theorem~\ref{j-classification}.

\section{Motivic decomposition of maximal orthogonal Grassmannians}
\label{section-grass}

\subsection{Statement of results}

For a non-degenerate quadratic form $q\colon V\rightarrow k$ with trivial discriminant let $\mathrm{OGr}(q)$ denote the variety of maximal totally isotropic subspaces of $V$. If $q$ is odd dimensional, denote $X=\mathrm{OGr}(q)$, and if $q$ is even-dimensional denote by $X$ one of the (isomorphic) connected components of $\mathrm{OGr}(q)$ (see~\cite[Section~86]{EKM}). Since $X$ is generically split, the decomposition of its Morava motive can be described in terms of the Hopf algebra structure of $\mathrm K(n)^*(\mathrm{SO}_m)$ by~\cite[Theorem~5.7]{PShopf}, cf.~\cite[Corollary~6.14]{LPSS}. In particular, for a generic quadratic form $q$ we will determine the $\mathrm K(n)$-motivic decomposition of $X$.

More generally, if $\mathrm{dim}\,q=2l+1$ or $\mathrm{dim}\,q=2l$, following~\cite[Definition~5.11]{V-gras} we define $J(q)$ as a subset of $\{1,\ldots,l\}$ such that $i\in J(q)$ iff  the image of $e_i$ in $H^*_{\mathrm{Ch}}(q)$ is $0$ (see Definition~\ref{j-inv-def}). By~\cite[Main Theorem~5.8]{V-gras}, one can recover 
$$
H^*_{\mathrm{Ch}}(q)=\mathrm{Ch}^*(\mathrm{SO}_{\mathrm{dim}\,q})/(e_i\mid i\in J(q)).
$$

We will need the theorem of the first and the third authors~\cite[Theorem~5.7]{PShopf} in the case $A^*=\mathrm K(n)^*$, $G=\mathrm{SO}_m$ (cf. also~\cite[Corollary~6.14]{LPSS}).

\begin{tm}[Theorem~5.7~\cite{PShopf}]
\label{two-layers}
For $n\in\mathbb N\setminus0$ let $q$ be a quadratic form of dimension $m$ with trivial discriminant. Let $X$ denote the connected component of its maximal orthogonal Grassmannian. Then there exists a $\mathrm K(n)$-motive $\mathcal R$ such that the $\mathrm K(n)$-motive of $X$ decomposes as:
\begin{equation}\label{decompGB}
\mathcal{M}_{\mathrm K(n)}(X)\simeq\bigoplus_{i\in\mathcal{I}}\mathcal{R}\{i\}
\end{equation}
for some multiset of non-negative integers $\mathcal{I}$. Moreover, the rank of $\mathcal R$ equals the rank of $H^*_{\mathrm K(n)}(q)$.
\end{tm}

Observe that the the rank of $\mathcal R$ is determined by $J(q)$ by Corollary~\ref{j-ch-k}. Using the developed techniques we can now prove a more precise version of the above result.

\begin{tm}
\label{refined-layers}
In the notation of Theorem~\ref{two-layers}, if $m\leq2^{n+1}-2$ or $2^n-1\in J(q)$ then $\mathcal R$ is indecomposable. If $m\geq2^{n+1}-1$ and $2^n-1\not\in J(q)$ then $\mathcal R$ decomposes as a sum of two non-isomorphic indecomposable motives of the same rank.
\end{tm}

As an immediate corollary, we obtain the following result.

\begin{cl}
\label{application}
In the notation of Theorem~\ref{two-layers}, assume additionally that
$$
J(q)\cap\{1,\ldots,2^n-1\}=\emptyset
$$
{\rm(}e.g., a generic quadratic form{\rm)}. Then the following holds.
\begin{enumerate}
\item
The $\mathrm K(n)$-motive $\mathcal M_{\mathrm K(n)}(X)$ of $X$ is indecomposable for $m\leq2^{n+1}-2$.
\item
The $\mathrm K(n)$-motive $\mathcal M_{\mathrm K(n)}(X)$ of $X$ has $2^{\lfloor\frac{m-1}{2}\rfloor-2^n+2}$ indecomposable summands of rank $2^{2^n-2}$ for $m\geq2^{n+1}-1$.
\end{enumerate}
\end{cl}

\begin{rk}
Compare the above corollary with~\cite[Corollary~6.14]{LPSS}. We proved there that the motive of $X$ is decomposable for $m\geq2^{n+1}+1$. We could not prove at that time, however, that for $m=2^{n+1}-1,2^{n+1}$, the motive of $X$ is also decomposable. 

In the case of generic $q$ we also did not claim that the summands found in~\cite[Corollary~6.14]{LPSS} are indecomposable. In fact they decompose as a sum of two motives of the same rank.
\end{rk}

\begin{proof}[Proof of Corollary~\ref{application}]

According to Theorem~\ref{two-layers} there exists a $\mathrm K(n)$-motive $\mathcal R$ such that the $\mathrm K(n)$-motive of $X$ decomposes as:
\begin{equation*}
\mathcal{M}_{\mathrm K(n)}(X)\simeq\bigoplus_{i\in\mathcal{I}}\mathcal{R}\{i\}
\end{equation*}
for some multiset of integers $\mathcal{I}$, moreover, the rank of $R$ equals the rank of $H_{\mathrm K(n)}(q)$.

By Corollary~\ref{j-ch-k} we conclude that 
$$
H_{\mathrm K(n)}(q)=H_{\mathrm{Ch}}(q)[v_n^{\pm1}]/(\overline e_i\mid i\geq 2^n)=\mathrm K(n)(\mathrm{SO}_m),
$$
where the last equality follows from the assumption on the $J$-invariant of $q$. 

We can now compare the ranks of $\mathrm K(n)(\mathrm{SO}_m)$ and $\mathrm K(n)(\overline X)$ to determine the cardinality of $\mathcal I$. By~\cite[Theorem~6.13]{LPSS}, the rank of $\mathrm K(n)(\mathrm{SO}_m)$ equals $2^{\lfloor\frac{m-1}{2}\rfloor}$ for $m\leq 2^{n+1}$, and $2^{2^n-1}$ for $m>2^{n+1}$. On the other hand, the rank of $\mathrm K(n)(\overline X)$ always equals $2^{\lfloor\frac{m-1}{2}\rfloor}$. This implies that $|\mathcal I|=1$ for $m\leq 2^{n+1}$, and $|\mathcal I|=2^{\lfloor\frac{m-1}{2}\rfloor-2^n+1}$ for $m>2^{n+1}$.

It remains to use that $\mathcal R$ is indecomposable for $m\leq 2^{n+1}-2$, and that $\mathcal R$ has two indecomposable summands of the same rank for $m>2^{n+1}-2$ by Theorem~\ref{refined-layers}.
\end{proof}

The next section will be devoted to the proof Theorem~\ref{refined-layers}.

\subsection{Idempotents in $\mathrm K(n)^*(\mathrm{SO}_m)^{\vee}$}

Dualizing Proposition~\ref{co-mult-ck}, we obtain the following description of the Hopf algebra $\mathrm{CK}(n)^*(\mathrm{SO}_m)^{\vee}$ for $m\leq 2^{n+1}$.

\begin{prop}
\label{idem-small}
In the notation of Proposition~\ref{answer-dual}, for an odd $m\leq 2^{n+1}-1$ the natural map 
$$
\mathrm{CK}(n)^*(\mathrm{SO}_{m})^\vee\rightarrow\mathrm{CK}(n)^*(\mathrm{SO}_{2^{n+1}-1})^\vee
$$
coincides with the inclusion of a sub-algebra generated by $\gamma_t(\alpha_{2i-1})$ for $1\leq i\leq r$ and $1\leq t<2^{d_i}$, where $r=\lfloor\frac{m+1}{4}\rfloor$ and $d_i=\left\lfloor\log_2\left(\frac{m-1}{2i-1}\right)\right\rfloor$. In particular, 
$$
\mathrm{CK}(n)^*(\mathrm{SO}_{m})^\vee\cong\bigotimes_{i=1}^{r}\Gamma_{d_i}(\alpha_{2i-1})
$$ 
is an isomorphism of co-algebras, and the multiplication table of $\mathrm{CK}(n)^*(\mathrm{SO}_{m})^\vee$ can be deduced from the identities~{\rm(}\ref{raz-d}{\rm)} and~{\rm(}\ref{dvas-d}{\rm)} of Proposition~\ref{answer-dual}. For $m$ even there is an isomorphism of Hopf algebras
$$
\mathrm{CK}(n)^*(\mathrm{SO}_{m-1})^\vee\cong\mathrm{CK}(n)^*(\mathrm{SO}_{m})^\vee.
$$
\end{prop}

As a corollary we obtain an analogous result for the periodic algebraic Morava K-theories $\mathrm{K}(n)^*(\mathrm{SO}_m)^{\vee}$ for $m\leq 2^{n+1}$. Therefore, to determine the idempotents of $\mathrm{K}(n)^*(\mathrm{SO}_m)^{\vee}$ for all $m$ it is enough to consider only the case $m=2^{n+1}-1$.

\begin{prop}
\label{idem-answer}
In the notation of Proposition~\ref{answer-dual}, the only non-trivial idempotents in $\mathrm{K}(n)^*(\mathrm{SO}_{2^{n+1}-1})^{\vee}$ are 
$$
v_n^{-1}\alpha_{2^n-1}\ \text{ and }\  1-v_n^{-1}\alpha_{2^n-1}.
$$
\end{prop}

\begin{proof}
Let $e$ be an idempotent in $\mathrm{K}(n)^*(\mathrm{SO}_{2^{n+1}-1})^{\vee}$. Decompose it in the basis of monomials
$$
\gamma_{t_1}(\alpha_{1})\gamma_{t_3}(\alpha_{3})\ldots\gamma_{t_{2^{n-1}}}(\alpha_{2^n-1})
$$
where $0\leq t_i<2^{k_i}$, and observe that
$$
\big(\gamma_{t_1}(\alpha_{1})\gamma_{t_3}(\alpha_{3})\ldots\gamma_{t_{2^{n-1}}}(\alpha_{2^n-1})\big)^2=
\begin{cases}
v_n^k\,\alpha_{2h_1-1}\ldots\alpha_{2h_k-1},&\text{if}\ t_i\in\{0,2^{k_i-1}\}\ \forall\,i,\\
0,&\text{elsewhere}.
\end{cases}
$$
Since $e^2=e$, we conclude that $e$ can contain only monomials of the form
$$
\alpha_{2h_1-1}\ldots\alpha_{2h_k-1}
$$
(with some coefficients from $\mathbb F_2[v_n^{\pm1}]$) in its decomposition. However,
$$
\big(\alpha_{2h_1-1}\ldots\alpha_{2h_k-1}\big)^2=
\begin{cases}
v_n^k\,\alpha_{2j_1-1}\ldots\alpha_{2j_k-1},&\text{if}\ k_{h_i}=1\ \forall\,i,\\
0,&\text{elsewhere}.
\end{cases}
$$
The condition $k_{h_i}=1$ is equivalent to $2^{n-2}<h_i\leq 2^{n-1}$, and the fact $e^2=e$ implies that $e$ contains only monomials $\alpha_{2h_1-1}\ldots\alpha_{2h_k-1}$ with $2^{n-2}<h_i\leq 2^{n-1}$ in its decomposition. Recall that
$$
\alpha_{2h_i-1}^2=v_n\alpha_{2j_i-1}
$$
for $2j_i-1=2(2h_i-1)-2^n+1$, and $2j_i-1$ cannot be greater than $2h_i-1$ for $2^{n-2}<h_i\leq 2^{n-1}$. Moreover, $2j_i-1=2h_i-1$ can only happen for $2h_i-1=2^n-1$. Then assume that $e$ contains a monomial $\alpha_{2h_1-1}\ldots\alpha_{2h_k-1}$ in its decomposition distinct from $1$ and $\alpha_{2^n-1}$, and take the greatest such monomial in the lexicographical order. Then $e^2$ cannot contain this monomial, which provides a contradiction.

Then $e$ is a linear combination of $1$ and $\alpha_{2^n-1}$, and it remains to observe that $e^2=e$ implies that $1$ can only have the coefficient $0$ or $1$, and $\alpha_{2^n-1}$ can only have the coefficient $0$ or $v_n^{-1}$ in the decomposition of $e$.
\end{proof}

As a corollary of Propositions~\ref{idem-small} and~\ref{idem-answer}, we obtain the following result.

\begin{cl}
\label{cl-idem-answer}
The ring $\mathrm{K}(n)^*(\mathrm{SO}_{m})^{\vee}$ does not contain non-trivial idempotents for $m\leq2^{n+1}-2$.
\end{cl}

For a quadratic form $q$ of dimension $m$ with trivial discriminant, it follows from Corollary~\ref{j-ch-k}, that 
$$
H^*_{\mathrm K(n)}(q)=\mathbb F_2[v_n^{\pm1}][e_1,e_3,\ldots,e_{2r-1}]/(e_1^{2^{a_1}},\ldots,e_{2r-1}^{2^{a_r}})
$$
for $0\leq a_i\leq k_i$ determined by $J(q)$, and $r=\mathrm{min}\left(\left\lfloor\frac{m+1}{4}\right\rfloor,\,2^{n-1}\right)$. Dually, $H^*_{\mathrm K(n)}(q)^\vee$ can be identified with a subalgebra $\bigotimes_{i=1}^{r}\Gamma_{a_i}(\alpha_{2i-1})$ in $\mathrm K(n)(\mathrm{SO}_m)^\vee$. In particular, we get the following corollary.

\begin{cl}
\label{h-idem}
In the notation above, if $m\leq2^{n+1}-2$ or $2^n-1\in J(q)$ then $H^*_{\mathrm K(n)}(q)^\vee$ does not contain non-trivial idempotents. If $m\geq2^{n+1}-1$ and $2^n-1\not\in J(q)$ then the only non-trivial idempotents in $\mathrm{K}(n)^*(\mathrm{SO}_{2^{n+1}-1})^{\vee}$ are 
$v_n^{-1}\alpha_{2^n-1}$  and  $1-v_n^{-1}\alpha_{2^n-1}$.
\end{cl}
\begin{proof}
By Proposition~\ref{idem-answer}, $H^*_{\mathrm K(n)}(q)^\vee$ contains non-trivial idempotents if and only if $\alpha_{2^n-1}\in H^*_{\mathrm K(n)}(q)^\vee$, i.e., $e_{2^n-1}$ does not map to zero in $H^*_{\mathrm K(n)}$. The latter exactly means that $2^n-1\not\in J(q)$.
\end{proof}

Now we are ready to finish the proof of Theorem~\ref{refined-layers}. 

\begin{proof}[Proof of Theorem~\ref{refined-layers}]

According to~\cite[Theorem~5.7]{PShopf} there is a one-to-one correspondence between motivic decompositions of $\mathcal R$ in the category of $\mathrm K(n)$-motives and direct sum decompositions of $H^*_{\mathrm K(n)}(q)^{\vee}$ as a module over itself. However, $H^*_{\mathrm K(n)}(q)^{\vee}$ is indecomposable as a module over itself for $m\leq 2^{n+1}-2$ or $2^n-1\in J(q)$ by Corollary~\ref{h-idem}, therefore $\mathcal R$ is also indecomposable in these cases.

If $m\geq2^{n+1}-1$ and $2^n-1\not\in J(q)$ then the only non-trivial idempotents in $H^*_{\mathrm K(n)}(q)^{\vee}$ are 
$v_n^{-1}\alpha_{2^n-1}$  and  $1-v_n^{-1}\alpha_{2^n-1}$. The corresponding modules $H_1=v_n^{-1}\alpha_{2^n-1}H^*_{\mathrm K(n)}(q)^{\vee}$  and  $H_2=(1-v_n^{-1}\alpha_{2^n-1})H^*_{\mathrm K(n)}(q)^{\vee}$ have equal ranks, but they are not isomorphic as $H^*_{\mathrm K(n)}(q)^{\vee}$-modules ($\alpha_{2^n-1}$ acts as the identity on the first one, and as zero on the second one).

Therefore, the motive $\mathcal R$ also decomposes as a sum $\mathcal R\cong\mathcal R_1\oplus\mathcal R_2$ for indecomposable $\mathcal R_i$, $i=1,2$. It remains to observe that these summands cannot be isomorphic. Indeed, the realisation $\alpha_\star$ of the isomorphism $\alpha\in\mathrm K(n)^*(X\times X)$ between $\mathcal R_i$ defines an isomorphism of $H_{\mathrm K(n)}(q)^{\vee}$-modules between $H_i$ by~\cite[Theorem~4.14]{PShopf}, which is impossible.
\end{proof}

\begin{rk}
We remark that for $m\geq2^{n+1}-1$, the $\mathrm K(n)$-motive of $X$ detects whether $2^n-1\in J(q)$ in the following way. If all indecomposable summands of $\mathcal M_{\mathrm K(n)}(X)$ are isomorphic to each other up to a shift, then $2^n-1\in J(q)$. However, if there are two indecomposable summands of $\mathcal M_{\mathrm K(n)}(X)$ which are not isomorphic to each other up to a shift, then $2^n-1\not\in J(q)$.
\end{rk}

\section*{Appendices}
\renewcommand{\thesection}{A1}

\renewcommand{\thesection}{A1}
\section{Algebraic part of $\mathrm{H}_*\big(\mathrm{SO}(2^{n+1}-1);\,\mathbb F_2\big)$}
\label{app-alg-part}

In this appendix we will give a proof of Proposition~\ref{quotient-ch}. By~\cite[Remark after Theorem~6]{Kac}, the natural map 
$$
\mathrm{Ch}^*(\mathrm{SO}_{m})\rightarrow\mathrm{H}^*(\mathrm{SO}(m);\,\mathbb F_2)
$$
is injective and its image coincides with the set of squares in $\mathrm{H}^*(\mathrm{SO}(m);\,\mathbb F_2)$. In Proposition~\ref{quotient-ch} we dualized this statement. To give a detailed proof we will use the following result.

\begin{lm}
\label{square-dual}
For an $\mathbb F_2$-algebra $H$ consider the switch map $\tau\colon H\otimes H\rightarrow H\otimes H$ sending $a\otimes b$ to $b\otimes a$. Let $H\otimes H^\tau$ denote the submodule of $\tau$-invariant elements, and $H\otimes H_\tau$ the module of $\tau$-coinvariants, i.e., the kernel and the cokernel of the map $\mathrm{id}+\tau$. Consider the map
$$
\mathrm{sq}\colon H\otimes H^\tau\rightarrow H
$$
given as a composition of the natural inclusion $H\otimes H^\tau\hookrightarrow H\otimes H$ followed by the multiplication $\mathrm m\colon H\otimes H\rightarrow H$. The image of $\mathrm{sq}$ coincides with the submodule of $H$ generated by squares.

Then the dual map to a switch $\tau^\vee$ coincides with the switch on $H^\vee\otimes H^\vee$, and the dual map to $\mathrm{sq}$,
$$
\mathrm{sq}^\vee\colon H^\vee\rightarrow H^\vee\otimes H^\vee_\tau,
$$
is given by the class of the co-multiplication $\Delta=\mathrm m^\vee$.
\end{lm}
\begin{proof}
Take an additive basis $\{e_i\}$ of $H$, and let $\{f_i\}$ denote the dual basis of $H^\vee$. Then
$$
\langle\tau^\vee(f_i\otimes f_j),\,e_h\otimes e_k\rangle=\langle f_i\otimes f_j,\,e_k\otimes e_h\rangle,
$$
therefore $\tau^\vee(f_i\otimes f_j)=f_j\otimes f_i$, in particular, 
$$
(H\otimes H^\tau)^\vee=\mathrm{Coker}\,\tau^\vee=H^\vee\otimes H^\vee_\tau.
$$
The last claim is clear.
\end{proof}

Let $\mathrm{Ch}^*(\mathrm{SO}_m)^\vee$ denote the Hopf algebra dual to $\mathrm{Ch}^*(\mathrm{SO}_m)$. 

\begin{prop}[Proposition~\ref{quotient-ch}]
In the notation of Subsection~\ref{top-morava}, the Hopf algebra $\mathrm{Ch}^*(\mathrm{SO}_{2^{n+1}-1})^\vee$ is isomorphic to the quotient of $\mathrm{H}_*(\mathrm{SO}(2^{n+1}-1);\,\mathbb F_2)$ modulo the ideal generated by $\beta_{2i-1}$ for all $1\leq i\leq 2^n-1$.
\end{prop}
\begin{proof}
Denote $H=\mathrm{H}^*(\mathrm{SO}(2^{n+1}-1);\,\mathbb F_2)$ and $H^\vee=\mathrm{H}_*(\mathrm{SO}(2^{n+1}-1);\,\mathbb F_2)$ its dual.

Consider the right exact sequence
$$
\xymatrix{
H\otimes H^\tau\ar[rr]^{\ \ \ \mathrm{sq}}&&H\ar[rr]^{\mathrm{coker(sq)}}&&C\ar[rr]&&0
}
$$
where $\mathrm{sq}$ is the map defined in Lemma~\ref{square-dual}. Then $C^\vee$ is the kernel of the map $\mathrm{sq}^\vee$. 

By Lemma~\ref{square-dual} we easily see that $\beta_{2i-1}$ lies in the kernel of $\mathrm{sq}^\vee$ for all $i$, $1\leq i\leq 2^n-1$: in fact, $\beta_{2i-1}$ are primitive, and 
$$
\Delta(\beta_{2i-1})=(\mathrm{id}+\tau)(1\otimes\beta_{2i-1})
$$ 
vanishes in $H^\vee\otimes H^\vee_\tau$.

Since $\mathrm{Ch}^*(\mathrm{SO}_{2^{n+1}-1})$ coincides with $\mathrm{Im}(\mathrm{sq})$, we conclude that 
$$
\xymatrix{
\mathrm{Ch}^*(\mathrm{SO}_{2^{n+1}-1})^\vee=\big(\mathrm{Im}(\mathrm{sq})\big)^\vee&&H^\vee\ar@{->>}[ll]&&\ \ C^\vee\ar@{>->}[ll]_{\mathrm{coker(sq)}^\vee}
}
$$
is exact, in particular, the kernel of the map $H^\vee\twoheadrightarrow\mathrm{Ch}^*(\mathrm{SO}_{2^{n+1}-1})^\vee$ contains  $\beta_{2i-1}$, $1\leq i\leq 2^n-1$.

However, the natural map $\mathrm{Ch}^*(\mathrm{SO}_{2^{n+1}-1})\hookrightarrow H$ is a morphism of algebras and co-algebras, therefore the dual map is also a morphism of algebras and co-algebras. Therefore all multiples of $\beta_{2i-1}$ are also mapped to zero in $\mathrm{Ch}^*(\mathrm{SO}_{2^{n+1}-1})^\vee$.

Comparing the dimensions we obtain the claim. More precisely, the set of monomials
$$
\{\gamma_{2t_1}(\beta_{1})\gamma_{2t_3}(\beta_{3})\ldots\gamma_{2t_{2^{n-1}}}(\beta_{2^n-1})\mid0\leq t_i<2^{k_i}\}
$$
(where $k_i$ are given by~(\ref{ki}) and $\gamma_0(\beta_{2i-1})$ stands for $1$) has the same cardinality as the dimension of $\mathrm{Ch}^*(\mathrm{SO}_{2^{n+1}-1})$, and the images of these monomials generate $\mathrm{Ch}^*(\mathrm{SO}_{2^{n+1}-1})^\vee$. 
\end{proof}

\renewcommand{\thesection}{A2}
\section{Bi-ideals in $\mathrm{Ch}^*(\mathrm{SO}_m)$}
\label{app-milnor-moore}

In this appendix we give a proof of Lemma~\ref{hopf-chow}. This result is well-known and included only for the convenience of the reader. We recall the following classical theorem.

\begin{tm}[Milnor--Moore]
\label{milnor-moore-thm}
If $H$ is a non-negatively graded, primitively generated, commutative, finite dimensional bi-algebra over $\mathbb F_p$, with $H_0 = \mathbb F_p$, then there is an isomorphism of bi-algebras
$$
H\cong\mathbb F_p[y_1,\ldots,y_t]/(y_1^{p^{l_1}}, \ldots,y_t^{p^{l_t}})
$$ 
for some $l_i$, where $y_i$ are primitive.
\end{tm}
In fact, a much more general theorem is proven in~\cite[Theorem~7.16]{MM}. Now we can give a proof of Lemma~\ref{hopf-chow}.

\begin{lm}[Lemma~\ref{hopf-chow}]
Let $H=\mathbb F_{2}[e_1,\,e_3,\ldots,e_{2r-1}]/(e_{1}^{2^{k_1}},e_{3}^{2^{k_2}},\ldots,e_{2r-1}^{2^{k_r}})$ be a commutative, graded Hopf algebra with $e_{2i-1}$ primitive of degree $2i-1$. Then any {\rm(}homogeneous{\rm)} bi-ideal $I$ of $H$ coincides with the ideal
$$
\big(e_1^{2^{a_1}},\,e_3^{2^{a_2}},\ldots,e_{2r-1}^{2^{a_{r}}}\big)
$$
for some $0\leq a_{i}\leq k_i$.
\end{lm}
\begin{proof}
Since $H/I$ is still a non-negatively graded commutative co-commutative Hopf algebra with $(H/I)_0=\mathbb F_2$, it is isomorphic to
$$
H/I\cong\mathbb F_{2}[\overline y_1,\,\overline y_2,\ldots,\overline y_s]/(\overline y_{1}^{2^{l_1}},\overline y_{2}^{2^{l_2}},\ldots,\overline y_{s}^{2^{l_s}})
$$
for some primitive $\overline y_i$ and some $l_i$ by Theorem~\ref{milnor-moore-thm}. Take arbitrary pre-images $y_i\in H$ of $\overline y_i$. Since the images $\overline e_{2k-1}$ of $e_{2k-1}$ are primitive in $H/I$, they have the form
$$
\overline e_{2k-1}=\sum\overline y_{i_j}^{2^{b_j}}
$$  
for some $i_j$ and $b_j<l_{i_j}$ (see~\cite[Lemma~6.1\,a)]{LPSS}). Comparing the dimensions, we conclude that $b_i=0$ and $\mathrm{deg}(y_{i_j})=2k-1$. Let 
$$
a_{k}=\mathrm{max}(l_{i_j})
$$
for the above set of $i_j$. Then $e_{2k-1}-\sum y_{i_j}\in I$, and $y_{i_j}^{2^{a_{k}}}\in I$, therefore
$$
e_{2k-1}^{2^{a_{k}}}\in I.
$$
Now let 
$$
J=\big(e_1^{2^{a_1}},\,e_3^{2^{a_2}},\ldots,e_{2r-1}^{2^{a_r}}\big),
$$
and consider the map of Hopf algebras
$$
\varphi\colon H/J\twoheadrightarrow H/I.
$$
We will show that $\varphi$ is injective on the set of primitive elements. Indeed, primitive elements of $H/J$ have the form
$$
\sum e_{2k_h-1}^{2^{c_h}}+J
$$
for some set of $k_h$ and $c_h<a_{k_h}$. The image of $e_{2k_h-1}$ under $\varphi$ coincides with the sum $\sum\overline y_{i_j}$ where $\mathrm{deg}(y_{i_j})=2k_h-1$. Comparing degrees we conclude that the sets of $y_{i_j}$ for different $k_h$ have empty intersection. Therefore 
$$
\varphi\left(\sum e_{2k_h-1}^{2^{c_h}}+J\right)=0
$$
implies that $y_{i_j}^{2^{c_h}}=0$ for all $i_j=i_j(k_h)$ and all $k_h$. However, this implies that $c_h\geq\mathrm{max}(l_{i_j})=a_{k_h}$, which leads to a contradiction.

Finally, we conclude that $\varphi$ is injective by~\cite[Lemma~6.8]{LPSS}.
\end{proof}

\renewcommand{\thesection}{A3}
\section{Restrictions on $J$-invariant}
\label{restrictions}

The following statement combined with Theorem~\ref{j-ch-ck} provides certain restrictions on the Chow $J$-invariant of Vishik. However, these restrictions are not new.

\begin{prop}
\label{restrictions2}
In the notation of Proposition~\ref{co-mult-ck}, let $m\leq2^{n+1}$, and let $I$ be a homogeneous saturated bi-ideal of $\mathrm{CK}(n)^*(\mathrm{SO}_m)$. Assume that $e_{2^n-1-2k}\in I$. Then $e_{2^n-1-k}\in I$.
\end{prop}

\begin{proof}
Since 
$$
I=\big(e_1^{2^{j_1}},\,e_3^{2^{j_2}},\ldots,e_{2r-1}^{2^{j_r}}\big)\trianglelefteq\mathrm{CK}(n)^*(\mathrm{SO}_m).
$$
by Proposition~\ref{hopf-ck}, our assumption is equivalent to $j_{2^{n-1}-k}=0$. %Let $H=\mathrm{CK}(n)^*(\mathrm{SO}_m)$, and $I=J\cap H$. Then 
Since $I$ is a bi-ideal, one has
$$
\widetilde\Delta(e_{2^n-1-2k})=v_n\,e_{2^n-1-k}\otimes e_{2^n-1-k}+v_n^2Q\in I\otimes H+H\otimes I.
$$
This implies that
$$
\overline e_{2^n-1-k}\otimes\overline e_{2^n-1-k}\in\overline I\otimes\overline H+\overline H\otimes\overline I
$$
by Lemma~\ref{sublemma}, and therefore $\overline e_{2^n-1-k}\in\overline I$. However, since 
$$
\overline I=\big(\overline e_1^{2^{j_1}},\,\overline e_3^{2^{j_2}},\ldots,\overline e_{2r-1}^{2^{j_r}}\big)\trianglelefteq\overline H,
$$
we conclude that $j_t\leq d$ for $2^n-1-k=2^d(2t-1)$. Using Theorem~\ref{j-ch-ck} we conclude that $e_{2^n-1-k}\in J$.
\end{proof}

We will show now how to deduce the above restrictions from old results.

Let $J(q)$ denote Vishik's original $J$-invariant~\cite[Definition~5.11]{V-gras} of quadratic form $q$ of dimension $m$. By Theorem~\ref{j-ch-ck} we conclude $J(q)$ is a subset of $\{1,2,\ldots,\left\lfloor\frac{m-1}{2}\right\rfloor\}$ such that $i\in J(q)$ if and only if $e_{i}$ is rational in $\mathrm{CK}(n)(\mathrm{SO}_m)$.

Vishik proved in~\cite[Proposition~5.12]{V-gras} that if $i\in J(q)$ and $t\in\mathbb N\setminus0$ such that $i+t\leq \left\lfloor\frac{m-1}{2}\right\rfloor$ and ${i\choose t}\equiv1\mod2$, then $i+t\in J(q)$. In fact, he proved that $S^t(e_i)=e_{i+t}$ in this case, where $S^i$ denotes the $i$-th Steenrod operation, see~\cite[Chapter~XI and Chapter~XVI, Section~89]{EKM}.

\begin{prop}
In the notation above, let $k=2^{a_1}+2^{a_2}+\ldots+2^{a_t}$ where $a_1=\nu_2(k)$ and $a_{i+1}=\nu_2(k-2^{a_1}-\ldots-2^{a_i})$, $0\leq a_1<a_2<\ldots<a_t$. Then
$$
e_{2^n-1-k}=S^{2^{a_t}}(\ldots S^{2^{a_2}}(S^{2^{a_1}}(e_{2^n-1-2k}))\ldots)\in\mathrm{Ch}(\mathrm{SO}_m).
$$
\end{prop}
\begin{proof}

Let $\pi\colon\mathbb N\cup0\rightarrow 2^{\mathbb N\cup0}$ be the injection given by the base $2$ expansions. For any $n\in\mathbb N\cup0$, the set $\pi(n)$ consists of all those $m\in\mathbb N\cup0$ such that the base $2$ expansion of $n$ has $1$ in the $m$-th position. For any $i,\, n\in\mathbb N\cup0$, the binomial coefficient ${n\choose i}$ is odd if and only if $\pi(i)\subseteq\pi(n)$ (see~\cite[Lemma~78.6]{EKM}).

It is enough to show that ${2^n-1-2k+2^{a_1}+\ldots+2^{a_i}\choose 2^{a_{i+1}}}\equiv1\mod2$. Obviously, 
$$
\pi(2^{a_1}+\ldots+2^{a_i})=\{a_1,\ldots,a_i\},
$$ 
in particular, $\pi(k)=\{a_1,\ldots,a_t\}$, $\pi(2k)=\{a_1+1,\ldots,a_t+1\}$, and, obviously, $\pi(2^n-1)=\{0,1,\ldots, n-1\}$. In other words, identifying $\pi$ with the base $2$ expansion, we conclude that $\pi(2^n-1-2k)$ has $0$-s in positions $a_i+1$ and $1$-s in all other positions from $0$ to $n-1$. Then it is easy to see that $\pi(2^n-1-2k+2^{a_1}+\ldots+2^{a_i})$ has $1$ in position $a_{i+1}$, i.e., $\pi(2^{a_{i+1}})\subseteq\pi(2^n-1-2k+2^{a_1}+\ldots+2^{a_i})$.
\end{proof}

\section*{Acknowledgement}
We would like to thank sincerely Markus Land for his valuable comments on the subject of the article.

The first author was supported by the DFG grant AN 1545/4-1. The work of the second author was partially performed at the St. Petersburg Leonhard Euler International Mathematical Institute and supported by the Ministry of Science and Higher Education of the Russian Federation (agreement no. 075–15–2022–287) and BASIS foundation grant ``Young Russia Mathematics'', and partially performed at the Bar-Ilan University and supported by ISF grant 1994/20. The third author was supported by the Foundation for the Advancement of Theoretical Physics and Mathematics ``BASIS'' and the grant of the Government of the Russian Federation for the state support of scientific research carried out under the supervision of leading scientists, agreement 14.W03.31.0030 dated 15.02.2018.

\end{document}